\documentclass[10pt]{amsart}

\usepackage{geometry}
\usepackage{filecontents}

\usepackage{amsmath}
\usepackage{amssymb}
\usepackage{amsthm}
\usepackage{amsfonts, dsfont}
\usepackage{paralist}
\usepackage{graphics} 
\usepackage{epsfig} 
\usepackage{graphicx} 
\usepackage{multirow}
\usepackage{epstopdf}
\usepackage{epstopdf}
\usepackage{verbatim}
\usepackage{mathrsfs}
\usepackage{mathtools}
\usepackage{pstricks}
\usepackage{relsize}
\usepackage{tikz}
\usetikzlibrary{matrix}
\usepackage{subcaption}
\usepackage{pgfplots}
\usepackage{fixltx2e}
\usepackage{enumitem}
\usepackage{upgreek}
\usepackage{url}
\usepackage[colorlinks=true]{hyperref}
\hypersetup{urlcolor=blue, linkcolor=blue, citecolor=red}
\usepackage{cleveref}
\usepackage{bm}
\usepackage{appendix}
\usepackage{algorithm2e}
\usepackage{algorithmic}
\usepackage{caption}

\DeclareMathOperator*{\argmin}{argmin}
\DeclareMathOperator*{\argmax}{argmax}

\definecolor{ao(english)}{rgb}{0.0, 0.5, 0.0}
\newenvironment{thisnote}{\par\color{ao(english)}}{\par}

\usepackage{amssymb}

\allowdisplaybreaks

\numberwithin{equation}{section}

\theoremstyle{plain} 
\newtheorem{thm}{Theorem}[section]

\theoremstyle{remark}
\newtheorem{rem}{Remark}[section]

\definecolor{ForestGreen}{RGB}{34,139,34}
\definecolor{ao(english)}{rgb}{0.0, 0.5, 0.0}

\begin{document}

\title[multiscale Consensus-Based Optimization algorithm]{A multiscale Consensus-Based algorithm for multi-level optimization}
\thanks{
}



\author{Michael Herty}
\address{Michael Herty \newline
Institut f\"ur Geometrie und Praktische Mathematik, RWTH Aachen University, 
Templergraben 55, 52062 Aachen, Germany. \newline
Extraordinary Professor, Department of Mathematics and Applied Mathematics, University of Pretoria, Private Bag X20, Hatfield 0028, South Africa}
\email{\texttt{herty@igpm.rwth-aachen.de}}

\thanks{ The authors thank the Deutsche Forschungsgemeinschaft (DFG, German Research Foundation) for the financial support through 320021702 / GRK2326, 
and for the financial support through 442047500/SFB1481 within the projects B04 (Sparsity fördernde Muster in kinetischen Hierarchien), B05 (Sparsifizierung zeitabhängiger Netzwerkflußprobleme mittels diskreter Optimierung) and B06 (Kinetische Theorie trifft algebraische Systemtheorie). Support received funding from the European Union’s Horizon Europe research and innovation programme under the Marie Sklodowska-Curie Doctoral Network Datahyking (Grant No. 101072546). DK is partially supported by the EPSRC Standard Grant EP/T024429/1. YH has been supported by a Roth Scholarship from Imperial College London.
}

\author{Yuyang Huang}
\address{Yuyang Huang \newline 
Department of Mathematics, Imperial College London, South Kensington Campus SW72AZ London, UK
}
\email{\texttt{yuyang.huang21@imperial.ac.uk}}
\thanks{ 
}

\author{Dante Kalise}
\address{Dante Kalise \newline 
Department of Mathematics, Imperial College London, South Kensington Campus SW72AZ London, UK
}
\email{\texttt{d.kalise-balza@imperial.ac.uk}}
\thanks{
}

\author{Hicham Kouhkouh}
\address{Hicham Kouhkouh \newline 
Department of Mathematics and Scientific Computing, NAWI, University of Graz, 
Heinrichstra{\ss}e 36, 8010, Graz, Austria
}
\email{\texttt{hicham.kouhkouh@uni-graz.at}}
\thanks{}





\date{\today}

\begin{abstract}
A novel multiscale consensus-based optimization (CBO) algorithm for solving bi- and tri-level optimization problems is introduced. Existing CBO techniques are generalized by the proposed method through the employment of multiple interacting populations of particles, each of which is used to optimize one level of the problem. These particle populations are evolved through multiscale-in-time dynamics, which are formulated as a singularly perturbed system of stochastic differential equations. Theoretical convergence analysis for the multiscale CBO model to an averaged effective dynamics as the time-scale separation parameter approaches zero is provided. The resulting algorithm is presented for both bi-level and tri-level optimization problems. The effectiveness of the approach in tackling complex multi-level optimization tasks is demonstrated through numerical experiments on various benchmark functions. Additionally, it is shown that the proposed method performs well on min-max optimization problems, comparing favorably with existing CBO algorithms for saddle point problems.
\end{abstract}













\maketitle

\section{Introduction}

We are concerned with an important class of multi-level optimization problems where decision variables cascade throughout a hierarchy. As a prototypical example, we consider the bi-level problem formulated as
\begin{equation}\label{bilevel opt}
    \begin{aligned}
        &\min\limits_{x\in\mathbb{R}^n} \, F(x,y)\\
        & \quad\text{s.t.}\quad  y\in \argmin_{y\in\mathbb{R}^m} \, G(x,y),
    \end{aligned}
\end{equation}
where $F,G:\mathbb{R}^{n}\times \mathbb{R}^{m}\to \mathbb{R}$ for some $n,m\in \mathbb{N}$. A lower-level problem  $\min\limits_{y} G(x,y)$ is  embedded within an upper-level problem $\min\limits_{x} F(x,y)$. This bi-level structure is ubiquitous in operations research and planning problems \cite{whittaker2017spatial,calvete2010multiobjective,wang2020bi}. For example, in a decentralized supply chain with one manufacturer and one distributor, the  optimization of the manufacturer's production decisions (e.g., production quantities, pricing strategies) at the upper level and the optimization of the distributor's distribution decisions (e.g., order quantities, transportation routes) at the lower level are  interdependent, leading to a hierarchical structure characteristic of bi-level optimization problems \cite{achamrah2022bi,amirtaheri2017bi}. 
With rapid development of data analysis, bi-level optimization has become increasingly significant 
in the community of signal processing and machine learning. Recent applications includes image processing \cite{crockett2022bilevel,ochs2015bilevel,chen2020flexible}, coreset selection \cite{sun2022learning,borsos2020coresets,zhou2022probabilistic} and robust training of deep neural networks \cite{zuo2021adversarial,zhang2022revisiting,yang2021robust,zhang2022advancing,xue2021rethinking}.  For example, when training large-scale machine learning models, it is crucial to identify the most informative subset of data from a massive entire dataset. This process, known as coreset selection, involves selecting the most representative data samples to form the coreset as lower-level problem and minimizing the model training loss of the selected coreset as upper-level problem \cite{zhang2023introduction}.

Bi-level optimization has gained high attention in tackling sequential decision-making processes owing to their inherent hierarchical structures. However, this also presents considerable challenges when solving it.
Finding optimum for bi-level optimization is generally NP-hard, even for the simplest linear case \cite{besanccon2021complexity}. For linear or convex bi-level problems, existing methods include Karush-Kuhn-Tucker conditions \cite{hansen1992new} and penalty function approaches \cite{zhao1998penalty}. The
presence of  non-convexity and 
high dimensionality of potential problems further  increase
the computational complexity \cite{murty1985some}.  Therefore, heuristic methods have been proposed to solve bi-level optimization problems,  see for instance, Simulated Annealing \cite{anandalingam1989artificial},  Particle Swarm Optimization \cite{jiang2021research}
and Genetic Algorithms \cite{oduguwa2002bi}.

Beyond bi-level problems, multi-level problems 
have broader applicability in diverse areas, ranging from agricultural economics \cite{candler1981potential}, design engineering \cite{barthelemy1988improved}, to optimal actuator placement \cite{KKS18,EKMS}.  In particular, a tri-level optimization problem is formulated as  
\begin{equation}
\begin{aligned} 
\label{trilevel opt}
    & \min _{x\in\mathbb{R}^n} F\left(x, y, r\right)\\ 
    & \quad\text { s.t. } \quad y \in \underset{y\in\mathbb{R}^m}{\operatorname{argmin}}\; G\left(x, y, r\right) \\ 
    & \quad\quad\text { s.t. }\quad r \in \underset{r\in\mathbb{R}^p}{\operatorname{argmin}} \;E\left(x, y, r\right).
\end{aligned}
\end{equation}
However, there is limited literature regarding optimization methods for multilevel problems due to the overwhelming complexity of the multilevel hierarchy. Existing methods include  metaheuristic methods \cite{tilahun2012new}, fixed point type iterations \cite{iiduka2011iterative} and gradient methods for convex problems \cite{sato2021gradient,shafiei2024trilevel}.

In this paper, we are interested in designing a variant of  consensus-based optimization (CBO) algorithms for bi-level and tri-level optimization problems. The CBO method is  a multi-particle derivative-free optimization method, originally designed for solving global non-convex optimization problems in high dimensions \cite{pinnau2017consensus}. The rigorous convergence
analysis of CBO in the mean-field limit has been thoroughly studied over recent years \cite{carrillo2018analytical,fornasier2024consensus,fornasier2025pde,ha2020convergence,ha2022stochastic,huang2022mean,wang2025mathematical}, in particular also in \cite{huang2024uniform,huang2025self,huang2025faithful}, and in the time-discrete setting \cite{ha2021convergence}. Meanwhile, numerous extensions have been proposed, including applications in constrained optimization \cite{borghi2023constrained,fornasier2022anisotropic,bae2022constrained}, multi-objective optimization \cite{borghi2023adaptive}, stochastic optimization \cite{bonandin2024consensus, bonandin2025kinetic}, large-scale machine learning \cite{tsianos2012consensus} and variations with jump diffusion \cite{kalise2023consensus}, with control \cite{huang2024fast}, with truncated diffusion \cite{fornasier2025consensus}, etc.

Of particular interest, Huang, Qiu and Riedl \cite{huang2024consensus} proposed a CBO model (hereafter denoted by \texttt{SP-CBO}) consisting of two interacting populations of particles to find saddle points for min-max optimization problems. Their problem can be written as
\begin{equation}\label{minmax opt}
    \min\limits_{x}\max\limits_{y} F(x,y).
\end{equation}
It can be noted that this is a special case of bi-level optimization. Indeed, 
by choosing $G= - F$  in \eqref{bilevel opt}, then $y\in \argmin -F(x,y)$ is equivalent to $y \in \argmax F(x,y)$, and we have
\begin{equation*}
    \min\limits_{x}\max\limits_{y} F(x,y) 
    \quad \Leftrightarrow \quad
    \left[\;
    \begin{aligned}
        &\min\limits_{x} F(x,y)\\
        & \quad\text{s.t.}\quad  y\in \argmin_y -F(x,y).
    \end{aligned}
    \right.  
\end{equation*}

\paragraph{\bf Contributions} Our contributions are twofold. 
Firstly, we introduce a multiscale CBO algorithm for the optimization problems \eqref{bilevel opt} and \eqref{trilevel opt}. 
We also show how it can be adapted to multi-level optimization problems. 
The main idea relies on an interacting system of multiple populations of particles, each one aiming at optimizing one of the levels in \eqref{bilevel opt}, but running on different time scales. 
Besides the novel application of CBO algorithm, the multiple scaling in time which our model exhibits could be of its own interest, and certainly also be adapted to various other problems and applications. See the discussion in \S \ref{sec: MS CBO}.

Secondly, we show that this multiscale CBO algorithm can be simplified thanks to the particular structure of the dynamics, which is suitable for the application of the so-called averaging principle. We construct the resulting reduced dynamics, referred to as \textit{effective} CBO, as it is customary in homogenization; see \Cref{sec: averaging}. Then we show the convergence of our multiscale CBO model to the effective CBO. Indeed this can be seen as a model reduction technique since the dimension of the resulting dynamics (or the number of particles) decreases. Ultimately, these ideas are put into action in \texttt{Algorithm \ref{algo averaging}} for bi-level problems \eqref{bilevel opt}, and in \texttt{Algorithms \ref{algo: tri-level words} \& \ref{algo: tri-level}} for tri-level problems \eqref{trilevel opt}.

\paragraph{\bf Organization} 
The rest of the paper is structured  as follows. In \Cref{sec: design},   starting from the introduction of the standard CBO model, we construct the multiscale CBO model, discussing the main motivation and mathematical intuition behind it.  \Cref{sec: implement} formalizes a modification of the model presented in the previous section, enhancing its theoretical and numerical suitability for multi-level optimization. Additionally, we present the corresponding algorithm.
 \Cref{sec: generalize} generalizes the model to tackle multi-level optimization problems and presents an algorithm for tri-level optimization problems.  \Cref{sec: test} provides a numerical benchmark over a standardized class of tests, assessing the efficiency of our model. We conclude in  \Cref{sec: conclude} with a brief summary of the results and future perspectives. The appendix \ref{app: proof conv} contains the proof of Theorem \ref{thm: conv}. 

\section{The multiscale CBO method}\label{sec: design}

Starting from the standard CBO method, in this section we introduce a new multiscale CBO algorithm for bi-level optimization problems of the form (\ref{bilevel opt}).

\subsection{The standard CBO model (parametrized)} 
Let $x\in \mathbb{R}^{n}$ be fixed.
The standard CBO methods for a minimization problem 
\begin{equation}\label{eq: max}
    \min\limits_{y\in \mathbb{R}^{m}} G(x,y)
\end{equation}
employs a group of  $M\in\mathbb{N}$ particles $\{Y^{j}_{x,t} \}_{j=1}^M$, which evolves in time $t \in [0,T]$ with $T<+\infty$,  according to the stochastic differential equations (SDEs): 
\begin{equation}\label{eq: dyn Y}
\begin{aligned}
    &\text{d} Y^{j}_{x,t}  = - \lambda_{2} \left(Y^{j}_{x,t} - \mathbf{v}(\widehat{\rho}_{Y_{x,t}},x)\right) \, \text{d} t + \sigma_{2}D\left( Y^{j}_{x,t} - \mathbf{v}(\widehat{\rho}_{Y_{x,t}},x) \right)\,\text{d} W^{Y,j}_{x,t},
\end{aligned}
\end{equation}
where the operator $D(\cdot) = \text{diag}(\cdot)$ maps a $d$-dimensional vector $v$ to a diagonal matrix in $\mathbb{R}^{d\times d}$
whose components are the elements of the vector $v$, and $\lambda_2,\,\sigma_2$ are positive constants. Let us denote $Y_{x,t} := (Y^{1}_{x,t}, \dots, Y^{M}_{x,t}) \in \left(\mathbb{R}^{m}\right)^{M}$, and $\widehat{\rho}_{Y_{x,t}}$ its corresponding empirical measure 
\begin{equation}\label{eq: emp mea Y x}
    \widehat{\rho}_{Y_{x,t}} = \frac{1}{M}\,\sum\limits_{k=1}^{M} \delta_{Y^{k}_{x,t}}.
\end{equation} 
The weight function is given by
\begin{equation}\label{weight beta}
    \varpi_{\beta}(x,y) = \exp(-\beta G(x,y))
\end{equation}
for some $\beta>0$ constant. 
Then the consensus point is then defined as:
\begin{equation}\label{eq: consensus y}
    \mathbf{v}(\widehat{\rho}_{Y_{x,t}},x) =  \sum\limits_{k=1}^{M} Y^{k}_{x,t}\frac{\varpi_{\beta}(x,Y^{k}_{x,t})}{\sum\limits_{r=1}^{M}\varpi_{\beta}(x,Y^{r}_{x,t})} = \int_{\mathbb{R}^{m}} y\,\frac{\varpi_{\beta}(x, y) }{\big\| \varpi_{\beta}(x, \cdot) \big\|_{L^{1}\left(\widehat{\rho}_{Y_{x,t}}\right)}} \,\text{d} \widehat{\rho}_{Y_{x,t}}(y)
\end{equation}
with $\big\| \varpi_{\beta}(x, \cdot) \big\|_{L^{1}\left(\widehat{\rho}_{Y_{x,t}}\right)} = \int_{\mathbb{R}^{m}} \varpi_{\beta}(x, y)\,\text{d}\widehat{\rho}_{Y_{x,t}}(y)$.

\subsection{The coupled CBO model}\label{sec: coupled CBO}
Let us now consider $N$ particles $\{X^{i}_{\cdot}\}_{i=1}^{N}$ in $\mathbb{R}^{n}$, where $n,N$ are fixed integers. To each particle $X^{i}_{\cdot}\in \mathbb{R}^{n}$ with $1 \leq i \leq N$, we associate a set of $M$ particles $(Y^{1}_{X^{i},\cdot}, \dots, Y^{M}_{X^{i},\cdot}) \in \left(\mathbb{R}^{m}\right)^{M}$, where $Y^j_{X^{i},\cdot}\in\mathbb{R}^m$ for any $1\leq j \leq M$. The dynamics of each $Y^{j}_{X^{i},t}$ is governed by \eqref{eq: dyn Y}, where instead of $x$, we now have $X^{i}$, and evolves in time $t\in [0,T]$ with $T <+\infty$. 
\begin{figure}[h]
\centering
\includegraphics[width=1\textwidth]{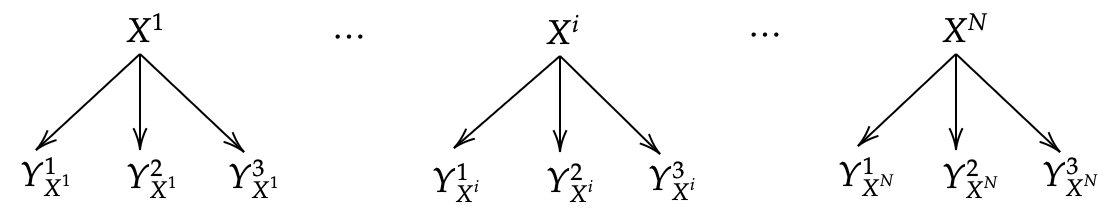}
\caption{Illustration of the particle systems in coupled CBO model with $M=3$ particles for $Y$-system. }
\label{A_0}
\end{figure}

The particles $X^{i}_{\cdot}$, with $1\leq i\leq N$, are designed for the minimization problem
\begin{equation*}
    \min\limits_{x\in \mathbb{R}^{n}} F\big(x, \xi(x)\big),
\end{equation*}
where $\xi(\cdot)$ is a function that \textit{approximately} captures the minimization problem \eqref{eq: max} in a way that we will soon make precise. 
These particles follow a CBO scheme
\begin{equation*}
\begin{aligned}
   & \text{d} X^{i}_{t}  =  -\lambda_{1} \left( X^{i}_{t} - {z}(\widehat{\rho}_{X_{t}}, \xi_{t}(X^{i}_{t}))\right)\,\text{d} t + \sigma_{1} D\left( X^{i}_{t} - {z}(\widehat{\rho}_{X_{t}}, \xi_{t}(X^{i}_{t}) )\right)\,\text{d} W^{X,i}_{t},
\end{aligned}
\end{equation*}
where $\lambda_1, \sigma_1>0$ are constants,  $X_{t} := (X^{1}_{t}, \dots, X^{N}_{t}) \in \left(\mathbb{R}^{n}\right)^{N}$, and $\widehat{\rho}_{X_{t}}$ its corresponding  empirical measure is
\begin{equation*}
    \widehat{\rho}_{X_{t}} = \frac{1}{N}\sum\limits_{i=1}^{N} \delta_{X^{i}_{t}}.
\end{equation*}
The weights functions are given by 
\begin{equation}\label{weight alpha}
    \omega_{\alpha}(x,y) = \exp(-\alpha F(x,y)),
\end{equation}
for some $\alpha>0$ constant. We can then define the consensus point as 
\begin{equation*}
    {z}(\widehat{\rho}_{X_{t}}, \xi_{t}(X^{i}_{t})) = \sum\limits_{k=1}^{N} X^{k}_{t}\frac{\omega_{\alpha}(X^{k}_{t}, \xi_{t}(X^{i}_{t}))}{\sum\limits_{r=1}^{N}\omega_{\alpha}(X^{r}_{t},\xi_{t}(X^{i}_{t}))} = \int_{\mathbb{R}^{n}} x \frac{\omega_{\alpha}(x,\xi_{t}(X^{i}_{t}))}{\|\omega_{\alpha}(\cdot\,,\xi_{t}(X^{i}_{t}))\|_{L^{1}\left(\widehat{\rho}_{X_{t}}\right)}}\,\text{d} \widehat{\rho}_{X_{t}}(x).
\end{equation*}
The function $\xi_{t}(\cdot)$ is the coupling term between the $x$- and $y$-population of particles. It is given as follows
\begin{equation*}
    \xi_{t}(x) = \sum\limits_{j=1}^{M} Y^{j}_{x,t} \frac{\varpi_{\beta}(x,Y^{j}_{x,t})}{\sum\limits_{\ell=1}^{M} \varpi_{\beta}(x,Y^{\ell}_{x,t})} = \int_{\mathbb{R}^{m}} y \frac{\varpi_{\beta}(x,y)}{\| \varpi_{\beta}(x,\cdot)\|_{L^{1}\left(\widehat{\rho}_{Y_{x,t}}\right)}}\,\text{d} \widehat{\rho}_{Y_{x,t}}(y),
\end{equation*}
where we recall the weight function $\varpi_{\beta}(\cdot\,,\cdot)$ is defined in \eqref{weight beta}. 
This is in fact the consensus of the group of $y$-particles associated with each $x$-particle at time $t$. In other words, given a particle $X^{i}_{t}$, $1\leq i \leq N$, the consensus point at time $t$ for its corresponding population of $y$-particles $Y_{X^{i},t} = (Y^{1}_{X^{i},t},\dots,Y^{M}_{X^{i},t})$ is  $\xi_{t}(X^{i}_{t})\in \mathbb{R}^{m}$ defined by
\begin{equation*}
\begin{aligned}
\label{eq:xi/v}
    \xi_{t}(X^{i}_{t}) & = \sum\limits_{j=1}^{M} Y^{j}_{X^{i},t} \frac{\varpi_{\beta}(X^{i}_{t},Y^{j}_{X^{i},t})}{\sum\limits_{\ell=1}^{M} \varpi_{\beta}(X^{i}_{t},Y^{\ell}_{X^{i},t})} = \int_{\mathbb{R}^{m}} y \frac{\varpi_{\beta}(X^{i}_{t},y)}{\| \varpi_{\beta}(X^{i}_{t},\cdot\,)\|_{L^{1}\left(\widehat{\rho}_{Y_{X^{i},t}}\right)}}\,\text{d} \widehat{\rho}_{Y_{X^{i},t}}(y),
\end{aligned}
\end{equation*}
hence we have $\, \xi_{t}(X^{i}_{t}) = \mathbf{v}(\widehat{\rho}_{Y_{X^{i},t}},X^{i}_{t}) \,$ where we recall $\mathbf{v}$ is defined in \eqref{eq: consensus y}. For simplicity of notation, we shall denote it by 
\begin{equation*}
    \mathbf{v}^{i}_{t} := \mathbf{v}(\widehat{\rho}_{Y_{X^{i},t}},X^{i}_{t}).
\end{equation*}
This means the consensus point ${z}$ takes the form
\begin{equation}\label{eq: consensus x empirical}
\begin{aligned}
    {z}(\widehat{\rho}_{X_{t}}, \xi_{t}(X^{i}_{t})) & = \sum\limits_{k=1}^{N} X^{k}_{t}\frac{\omega_{\alpha}(X^{k}_{t}, \mathbf{v}(\widehat{\rho}_{Y_{X^{i},t}},X^{i}_{t}))}{\sum\limits_{r=1}^{N}\omega_{\alpha}(X^{r}_{t},\mathbf{v}(\widehat{\rho}_{Y_{X^{i},t}},X^{i}_{t}))}\\
    & = \int_{\mathbb{R}^{n}} x \frac{\omega_{\alpha}(x,\mathbf{v}(\widehat{\rho}_{Y_{X^{i},t}},X^{i}_{t}))}{\|\omega_{\alpha}(\cdot\,,\mathbf{v}(\widehat{\rho}_{Y_{X^{i},t}},X^{i}_{t}))\|_{L^{1}\left(\widehat{\rho}_{X_{t}}\right)}}\,\text{d} \widehat{\rho}_{X_{t}}(x)\\
    & =: \mathbf{z}(\widehat{\rho}_{X_{t}}, \mathbf{v}^{i}_{t}).
\end{aligned}
\end{equation}
Therefore, the dynamics are given by
\begin{equation}\label{eq: dyn X}
\begin{aligned}
    & \text{d} X^{i}_{t} =  -\lambda_{1} \left( X^{i}_{t} - \mathbf{z}(\widehat{\rho}_{X_{t}}, \mathbf{v}^{i}_{t})\right)\,\text{d} t +\sigma_{1} D\left( X^{i}_{t} - \mathbf{z}(\widehat{\rho}_{X_{t}}, \mathbf{v}^{i}_{t}) \right)\,\text{d} W^{X,i}_{t}, 
\end{aligned}
\end{equation}
and 
\begin{equation}\label{v M i t}
    \mathbf{v}^{i}_{t} = \sum\limits_{j=1}^{M} Y^{j}_{X^{i},t} \frac{\varpi_{\beta}(X^{i}_{t},Y^{j}_{X^{i},t})}{\sum\limits_{\ell=1}^{M} \varpi_{\beta}(X^{i}_{t},Y^{\ell}_{X^{i},t})} = \int_{\mathbb{R}^{m}} y \frac{\varpi_{\beta}(X^{i}_{t},y)}{\| \varpi_{\beta}(X^{i}_{t},\cdot\,)\|_{L^{1}\left(\widehat{\rho}_{Y_{X^{i},t}}\right)}}\,\text{d} \widehat{\rho}_{Y_{X^{i},t}}(y).
\end{equation}

Note that the $X$-particles are interacting between each other, and each $X$-particle is also interacting with its corresponding $Y$-particles. But for any $1\leq i\leq N$, the particles $\{Y_{X^{i},\cdot}^j\}_{j=1}^M$ associated to  $X_\cdot^i$ do not interact with the particles $\{Y_{X^{k},\cdot}^j\}_{j=1}^M$ for any $k\neq i$.

\subsection{The multiscale CBO model}\label{sec: MS CBO}

The presence of a hierarchical structure in the bi-level optimization problem, or in the order of optimization for min-max problems, suggests that the $x$- and $y$-particles are not at the same decision level. 
Indeed, while the $y$-particles  perform optimization, the $x$-particles remain  frozen, waiting for the $y$-particles to reach a steady-state. Numerically, this means that for each element of the $x$-population, we update its corresponding $y$-population for several time steps before we update the $x$-population. In other words, the evolution of $y$-population is on a faster time-scale than the $x$-population. This distinction motivates the introduction of two time scales.  
 
The multi-time scaling of such coupled systems of SDEs can be modelled with \textit{singular perturbation}, that is, given some small parameter $0< \varepsilon \ll 1$, we consider the following dynamics:
\begin{equation}\label{eq: sys eps}
\left\{\;
\begin{aligned}
   & \text{d} X^{i}_{t} =  -\lambda_{1} \left( X^{i}_{t} - \mathbf{z}(\widehat{\rho}_{X_{t}}, \mathbf{v}^{i}_{t})\right)\,\text{d} t + \sigma_{1} D\left( X^{i}_{t} - \mathbf{z}(\widehat{\rho}_{X_{t}}, \mathbf{v}^{i}_{t}) \right)\,\text{d} W^{X,i}_{t}\\
   & \text{d} Y^{j}_{X^{i},t}  = - \lambda_{2}\frac{1}{\varepsilon} \left(Y^{j}_{X^{i},t} - \mathbf{v}^{i}_{t}\right) \, \text{d} t + \sqrt{\frac{1}{\varepsilon}}\sigma_{2}D\left( Y^{j}_{X^{i},t} - \mathbf{v}^{i}_{t} \right)\,\text{d} W^{Y,j}_{X^{i},t}\\
   & \text{Given } \, (X^{i}_{0}, Y^{j}_{i,0}), \quad \quad 1\leq j\leq M,\quad 1\leq i \leq N, \quad 0\leq t \leq T,
\end{aligned}
\right.
\end{equation}
where $\mathbf{v}^{i}_{t}$ defined in \eqref{v M i t} is a function of $(X^{i}_{t}, Y_{X^{i},t}^{1}, \dots \,, Y_{X^{i},t}^{M})$, and
and $\mathbf{z}(\cdot\,,\cdot\,)$ is defined in \eqref{eq: consensus x empirical}. 

The $x$-processes are referred to as the \textit{slow dynamics}, while the $y$-processes are the \textit{fast dynamics}. This denomination can be justified as follows. Let us change the time variable using $\tau = t/\varepsilon$. The system \eqref{eq: sys eps} becomes 
\begin{equation*}
\left\{\;
\begin{aligned}
   & \text{d} \widetilde{X}^{i}_{\tau} =  -\varepsilon\,\lambda_{1} \left( \widetilde{X}^{i}_{\tau} - \mathbf{z}(\widehat{\rho}_{\widetilde{X}_{\tau}}, \mathbf{v}^{i}_{\tau})\right)\,\text{d} \tau + \sqrt{\varepsilon}\,\sigma_{1} D\left( \widetilde{X}^{i}_{\tau} - \mathbf{z}(\widehat{\rho}_{\widetilde{X}_{\tau}}, \mathbf{v}^{i}_{\tau}) \right)\,\text{d} \widetilde{W}^{X,i}_{\tau}\\
   & \text{d} \widetilde{Y}^{j}_{\widetilde{X}^{i},\tau}  = - \lambda_{2}\left(\widetilde{Y}^{j}_{\widetilde{X}^{i},\tau} - \mathbf{v}^{i}_{\tau}\right) \, \text{d} \tau + \sigma_{2}D\left( \widetilde{Y}^{j}_{\widetilde{X}^{i},\tau} - \mathbf{v}^{i}_{\tau} \right)\,\text{d} \widetilde{W}^{Y,j}_{\widetilde{X}^{i},\tau}\\
   & \text{Given } \, (\widetilde{X}^{i}_{0}, \widetilde{Y}^{j}_{\widetilde{X}^{i},0}), \quad \quad 1\leq j\leq M,\quad 1\leq i \leq N,
\end{aligned}
\right.
\end{equation*}
where we use the notation $\widetilde{Z}_{\tau} = Z_{\varepsilon \tau}$, for $Z_{\cdot} = X^{i}_{\cdot}, Y^{j}_{X^{i},\cdot}$, and $\widetilde{W}^{X,i}_{\tau} = (1/\sqrt{\varepsilon}) W^{X,i}_{\varepsilon \tau}$, $ \widetilde{W}^{Y,j}_{i,\tau} =  (1/\sqrt{\varepsilon}) W^{Y,j}_{\widetilde{X}^{i},\varepsilon\tau}$. Thus it can be seen that when $\varepsilon$ is small, the processes $\widetilde{X}^{i}_{\tau}$ tend to be frozen compared to the evolution of  $\widetilde{Y}^{j}_{\widetilde{X}^{i},\tau}$. In other words, by the time  the processes $\widetilde{X}^{i}_{\tau}$ make a significant change in their evolution, the processes $\widetilde{Y}^{j}_{\widetilde{X}^{i},\tau}$ would have already reached their long-time steady state. This steady-state behavior should be captured by their invariant measure under suitable recurrence condition. The treatment (both in theory and practice) of two-scale dynamics of this type is a challenging task. This motivates the introduction of simplified dynamics obtained by passing to the limit $\varepsilon \to 0$ in \eqref{eq: sys eps}. Indeed, in this case, one  obtains a system consisting only of the \textit{slow} component of \eqref{eq: sys eps}, whose dynamics is averaged with respect to the invariant measure of the \textit{fast} component, as it will be explained in more detail in the next section. This phenomenon is known as the \textit{averaging principle} and is the core idea behind singular perturbations and homogenization. In the context of optimization, singular perturbations of SDEs have been used in \cite{chaudhari2018deep} which was then extended to a more general setting in \cite{bardi2023singular, bardi2024deep} where a control parameter have been introduced, playing the role of a controlled learning rate. 
Therein, the singularly perturbed dynamics played the role of an approximation of the \textit{Entropic (controlled) Stochastic Gradient Descent}. Another context in which multiscale methods and singularly perturbed dynamics (also referred to as fast--slow dynamics) arise is financial mathematics. We refer in particular to \cite{fouque2003multiscale,fouque2003singular} and to the book \cite{fouque2000derivatives} with the numerous references therein. Indeed, it appears that many financial models suffer from a fast volatility which motivates the development of \textit{averaging methods} in order to provide a reliable and trackable approximation of these financial signals. Similar problems motivated by physical phenomena are discussed in \cite[Appendix 2]{herty2024relaxation}, 
and in more details in the  books \cite{weinan2011principles, pavliotis2008multiscale}.

A particular advantage of the latter is that it also plays the role of a model reduction technique: from a system consisting of two components (the $x$- and $y$-particles), we obtain a system of a single component (the modified $x$-particles). On the other hand, a challenge associated to the averaging method is that it requires the (explicit) computation of the invariant measure of the \textit{fast} process (the $y$-particles). In practice, we shall bypass this obstacle by substituting the average with respect to the invariant measure (which is an average in space) by the average with respect to time; see \eqref{eq: time average}. This is possible due to ergodicity of the $y$-particles, as established by the celebrated Birkhoff's ergodic theorem (see, e.g., \cite[equation (6.4.2)]{pavliotis2008multiscale}, and more generally \cite{birkhoff1931proof,keane2006easy,da1996ergodicity}). 

Although the averaging principle has been intensively studied since decades, to the best of our knowledge, the system described in \eqref{eq: sys eps} does not satisfy the assumptions in the very recent developments of the averaging principle (see for example \cite{hong2023central, hong2022strong, rockner2021strong, li2022near} and the references therein). Further modifications are thus needed, which is the object of the next section. 

\section{Implementation and convergence of the CBO model with effective dynamics}\label{sec: implement}

\subsection{The modified multiscale CBO model}\label{sec:modified CBO}
Before formalizing the modified multiscale CBO model, we shall introduce some notations. 
Given $R>0$ fixed, we define for $\forall\, v=(v_{k})_{1\leq k \leq d}\in \mathbb{R}^{d}$, two truncation functions 
\begin{equation*}
\begin{aligned}
    &  \phi_{R}(v) = \big(\min\{|v_{k}|,R\}\big)_{1\leq k \leq d}\, \in\mathbb{R}^{d},\\
    &  \psi_{R}(v) = \big(\psi_{R}(v)_{k}\big)_{1\leq k\leq d}, \quad \text{ where }\quad \psi_{R}(v)_{k} = \left\{\;
    \begin{aligned}
        & v_{k} \; , \quad \quad \text{ if } \; |v_{k}|\leq R\\
        & \frac{v_{k}}{|v_{k}|}R \; , \quad \text{ if } \; |v_{k}|>R.
    \end{aligned}
    \right.
\end{aligned}
\end{equation*}
Then, we define a matrix-valued function
\begin{equation*}
    \forall\, v\in \mathbb{R}^{d},\quad D^{\delta}_{R}(v) = \text{diag}(\delta+\phi_{R}(v)) \in \mathbb{R}^{d\times d},
\end{equation*}
where as before, $\text{diag}(\cdot)$ is the diagonal matrix whose components are the elements of the vector to which it is applied, e.g., the $k$-th element in the diagonal of $D^{\delta}_{R}(v)$ is $\delta + \min\{|v_{k}|, R\}>0$. In particular $D^{\delta}_{R}(\cdot)$ is a positive definite matrix. We shall keep the same notation $\phi_{R}, \psi_{R},D^{\delta}_{R}$ regardless of the dimension $d$, as it will be clear from the context. We then define a matrix-valued function
\begin{equation*}
    \forall\, M_{1},\dots,M_{k} \in \mathbb{R}^{d\times d},\quad \mathbb{D}(M_{1},\dots,M_{k}) = \text{Diag}(M_{1},\dots,  M_{k}),
\end{equation*}
where $\text{Diag}(\cdot)$ is the block-diagonal matrix whose diagonal blocks are the matrices $M_{1},\dots,M_{k}$. In particular, we define
\begin{equation*}
    \forall\, v^{1},\dots,v^{k} \in \mathbb{R}^{d},\quad \mathbb{D}^{\delta}_{R}(v^{1},\dots,v^{k}) = \mathbb{D}(D^{\delta}_{R}(v^{1}),\dots,D^{\delta}_{R}(v^{k})).
\end{equation*}
This is a block-diagonal matrix in $\mathbb{R}^{dk\times dk}$. Each of its $j$-th block ($1\leq j\leq k$)  is a diagonal matrix whose diagonal components are $\delta + \min\{|v_{i}^{j}|,R\}$ with $1\leq i \leq d$.

To keep the notation concise in what follows, let us denote the state vector of the particle system using the notation
\begin{equation*}
\begin{aligned}
    & \mathds{X}^{\varepsilon} = \big(\mathbf{X}_{i}\big)_{1\leq i \leq N} \in \mathbb{R}^{nN},\quad \text{ where } \mathbf{X}_i\in\mathbb{R}^{n},\\
    & \mathds{Y}^{\varepsilon} = (\mathbf{Y}_{i})_{1\leq i \leq N} \in \mathbb{R}^{mMN},\quad \text{ where } \mathbf{Y}_{i} = (Y^{j}_{\mathbf{X}_{i}})_{1\leq j \leq M} \in \mathbb{R}^{mM}, \quad \text{ and } Y^{j}_{\mathbf{X}_{i}} \in \mathbb{R}^{m}.
\end{aligned}
\end{equation*} 
The finite dimensional system of singularly perturbed SDEs is stated as follows:
\begin{equation}\label{ms_CBO}
\left\{\;
\begin{aligned}
    &\text{d} \mathds{X}^{\varepsilon}_{t} =  \mathscr{F}(\mathds{X}^{\varepsilon}_t,\mathds{Y}^{\varepsilon}_t)\,\text{d} t + \mathscr{G}(\mathds{X}^{\varepsilon}_t,\mathds{Y}^{\varepsilon}_t)\,\text{d} \mathds{W}^{1}_{t}\\
    &\text{d} \mathds{Y}^{\varepsilon}_{t} = \frac{1}{\varepsilon}\mathscr{B}(\mathds{X}^{\varepsilon}_t,\mathds{Y}^{\varepsilon}_t)\, \text{d} t + \frac{1}{\sqrt{\varepsilon}} \mathscr{H}  (\mathds{X}^{\varepsilon}_t,\mathds{Y}^{\varepsilon}_t)\,\text{d} \mathds{W}^{2}_t\\ 
    & \mathds{X}^{\varepsilon}_{0} = \mathbf{x}_0 \in \mathbb{R}^{nN}, \quad \text{ and } \quad \mathds{Y}^{\varepsilon}_{0} = \mathbf{y}_0 \in \mathbb{R}^{mMN},
\end{aligned}
\right.
\end{equation}
where $\mathds{W}^{1}_{\cdot}$ and $\mathds{W}^{2}_{\cdot}$ are $nN$, $mMN$-dimensional independent standard Brownian motions. 
We can now define the coefficients in \eqref{ms_CBO}. To do so, in the following subsections, we shall drop $\varepsilon$ in the notation and simply write $(\mathds{X}, \mathds{Y})$ instead of $(\mathds{X}^{\varepsilon}, \mathds{Y}^{\varepsilon})$.  

\subsubsection{The drift $\mathscr{B}$ in $\mathds{Y}$-process} \label{sec: drift fast}

The drift term of the fast process $\{\mathds{Y}_t\}_{t\geq0}$ is given by
\begin{equation*}
    \mathscr{B}(\mathds{X},\mathds{Y}) = \big( B(\mathbf{X}_{i},\mathbf{Y}_{i}) \big)_{1\leq i \leq N}, \quad \text{ where } \quad 
    B(\mathbf{X}_{i},\mathbf{Y}_{i}) = \big( b_{i}^{j}(\mathbf{X}_{i},\mathbf{Y}_{i}) \big)_{1\leq j \leq M}
\end{equation*}
and 
\begin{equation*}
\begin{aligned}
    b_{i}^{j}(\mathbf{X}_{i},\mathbf{Y}_{i}) 
    & = -\lambda_{2} \, \psi_{R_{2}} \left(\,Y_{\mathbf{X}_{i}}^j - \kappa \,\sum_{r=1}^M Y_{\mathbf{X}_{i}}^r \frac{\varpi_{\beta}\left(\mathbf{X}_i, Y_{\mathbf{X}_{i}}^r \right)}{\sum_{\ell=1}^M \varpi_{\beta}\left(\mathbf{X}_i, Y_{\mathbf{X}_{i}}^{\ell}\right)}\,\right)\in\mathbb{R}^{m},
\end{aligned}
\end{equation*}
where $\,\kappa\,$ is a positive constant that will be needed later in Theorem \ref{thm: conv}. It guarantees the validity of a recurrence condition which ultimately ensures ergodicity of the process, i.e. existence and uniqueness of its invariant measure (see point (iii) in the proof of Theorem \ref{thm: conv} in Appendix \ref{app: proof conv}, and also Remark \ref{rem:discussion} below). 
Therefore we have 
\begin{equation*}
    b_{i}^{j}(\mathbf{X}_{i},\mathbf{Y}_{i}) \in \mathbb{R}^{m}, \quad B(\mathbf{X}_{i},\mathbf{Y}_{i}) \in \mathbb{R}^{mM} \quad \text{ and } \quad \mathscr{B}(\mathds{X},\mathds{Y})\in \mathbb{R}^{mMN}.
\end{equation*}
Note that, using $\psi_{R_{2}}$, the drift functions $b_{i}^{j}$ take values in balls of radius $\lambda_{2}R_{2}$. 

\subsubsection{The diffusion $\mathscr{H}$ in $\mathds{Y}$-process}

The diffusion term of the fast process $\{\mathds{Y}_t\}_{t\geq0}$ is given by
\begin{equation*}
\begin{aligned}
    \mathscr{H} (\mathds{X},\mathds{Y}) & = \text{Diag}\left(\{H(\mathbf{X}_{i},\mathbf{Y}_{i})\}_{1\leq i \leq N} \right), \\
    \text{where } \; H(\mathbf{X}_{i},\mathbf{Y}_{i}) & = \sigma_{2} \,\mathbb{D}^{\delta_{2}}_{R_{2}}\left(\{h_{i}^{j}(\mathbf{X}_{i},\mathbf{Y}_{i})\}_{1\leq j \leq M}\right),
\end{aligned}
\end{equation*}
and, similar to $b_{i}^{j}$, we have
\begin{equation*}
\begin{aligned}
    h_{i}^{j}(\mathbf{X}_{i},\mathbf{Y}_{i}) 
    & =  Y_{\mathbf{X}_{i}}^j - \kappa\, \sum_{r=1}^M Y_{\mathbf{X}_{i}}^r \frac{\varpi_{\beta}\left(\mathbf{X}_i, Y_{\mathbf{X}_{i}}^r \right)}{\sum\limits_{\ell=1}^{M}\varpi_{\beta}\left(\mathbf{X}_i, Y_{\mathbf{X}_{i}}^\ell\right)}
    \in\mathbb{R}^{m}.
\end{aligned}
\end{equation*}
Therefore we have
\begin{equation*}
    h_{i}^{j}(\mathbf{X}_{i},\mathbf{Y}_{i})\in \mathbb{R}^{m}, \quad H(\mathbf{X}_{i},\mathbf{Y}_{i}) \in \mathbb{R}^{mM\times mM}, \quad \text{ and} \quad \mathscr{H}(\mathds{X},\mathds{Y}) \in \mathbb{R}^{mMN\times mMN}.
\end{equation*}
Using $\mathbb{D}^{\delta_{2}}_{R_{2}}$, the diffusion functions $H$ take values in balls of radius $\sigma_{2}(\delta_{2}+R_{2})$.

\subsubsection{The drift $\mathscr{F}$ in $\mathds{X}$-process}

The drift term of the slow process $\{\mathds{X}_t\}_{t\geq0}$ is given by 
\begin{equation*}
\begin{aligned}
    \mathscr{F}(\mathds{X},\mathds{Y})&  = \big(F_{i}(\mathds{X},\mathbf{Y}_{i})\big)_{1\leq i \leq N}\\
    \text{and }\; F_{i}(\mathds{X},\mathbf{Y}_{i}) & = -\lambda_{1}\, \psi_{R_{1}} \left( \mathbf{X}_{i} - \sum\limits_{r=1}^{N} \mathbf{X}_{r} \frac{\omega_{\alpha}\left(\mathbf{X}_{r},\frac{1}{\kappa}\mathbf{v}^{i}\right)}{\sum\limits_{\ell=1}^{N}\omega_{\alpha}\left(\mathbf{X}_{\ell},\frac{1}{\kappa}\mathbf{v}^{i}\right)} \right),
\end{aligned}
\end{equation*}
where we recall
\begin{equation}\label{consensus v i}
    \mathbf{v}^{i} = \sum\limits_{j=1}^{M} Y^{j}_{\mathbf{X}_{i}} \frac{\varpi_{\beta}(\mathbf{X}_{i},Y^{j}_{\mathbf{X}_{i}})}{\sum\limits_{\ell=1}^{M} \varpi_{\beta}(\mathbf{X}_{i},Y^{\ell}_{\mathbf{X}_{i}})}.
\end{equation}
Therefore we have
\begin{equation*}
    F_{i}(\mathds{X},\mathbf{Y}_{i}) \in \mathbb{R}^{n}\quad \text{ and } \quad \mathscr{F}(\mathds{X},\mathds{Y}) \in \mathbb{R}^{nN}.
\end{equation*}
Note that, using $\psi_{R_{1}}$, the drift functions $F_{i}$ take values in balls of radius $\lambda_{1}R_{1}$. 

\subsubsection{The diffusion $\mathscr{G}$ in $\mathds{X}$-process}

The diffusion term of the slow process $\{\mathds{X}_t\}_{t\geq0}$ is given by 
\begin{equation*}
    \mathscr{G}(\mathds{X},\mathds{Y}) = \sigma_{1}\, \mathbb{D}^{\delta_{1}}_{R_{1}}\left(\{ 
G_{i}(\mathds{X},\mathbf{Y}_{i}) \}_{1\leq i \leq N}\right)
\end{equation*}
and, similar to $F_{i}$, we have
\begin{equation*}
    G_{i}(\mathds{X},\mathbf{Y}_{i})  =\mathbf{X}_{i} - \sum\limits_{r=1}^{N} \mathbf{X}_{r} \frac{\omega_{\alpha}\left(\mathbf{X}_{r},\frac{1}{\kappa}\mathbf{v}^{i}\right)}{\sum\limits_{\ell=1}^{N}\omega_{\alpha}\left(\mathbf{X}_{\ell},\frac{1}{\kappa}\mathbf{v}^{i}\right)}.
\end{equation*}
Therefore we have
\begin{equation*}
    G_{i}(\mathds{X},\mathbf{Y}_{i}) \in \mathbb{R}^{n} \quad \text{ and } \quad \mathscr{G}(\mathds{X},\mathds{Y}) \in \mathbb{R}^{nN\times nN}.
\end{equation*}
Using $\mathbb{D}^{\delta_{1}}_{R_{1}}$, the diffusion function $\mathscr{G}$ takes values in a unit ball of radius $\sigma_{1}(\delta_{1}+R_{1})$.

\begin{rem}\label{rem:discussion} 
The following are some relevant observations regarding the proposed multiscale CBO model. 
\begin{enumerate}  
    \item Without loss of generality, one could consider $R_{1}=R_{2}=:R$, with $R<\infty$ sufficiently large, because the CBO algorithm works in practice over a compact domain, for example an Euclidean ball $\mathcal{B}_0(\bar{R})$ centered in $0$ with radius $\bar{R}$, which contains the initial position of the particles.
    Therefore, if the truncation parameter $R$ is chosen large enough, i.e., $R\gg \bar{R}$, the truncation function becomes $\phi_R(v) = (|v_{k}|)_{1\leq k \leq d}$ since the particles are unlikely to leave a neighborhood of $\mathcal{B}_0(\bar{R})$. This has been made precise in \cite[Proposition 2.3]{bayraktar2025uniform}. It is also worth mentioning that a CBO algorithm with a truncated noise has been studied in \cite{fornasier2024consensus}. 
    \item The parameters $\delta_1, \delta_2$ are chosen to be small constants. In particular, one could set  $\delta_{1} = \delta_{2}=:\delta>0$. They ensure the diffusion matrix remains positive definite, providing enough noise to prevent the dynamics from collapsing to one point very quickly.  Numerically, this still allows the particles to converge to a small neighborhood around the global optimizer.  
    \item In the present context, the parameter $\kappa$ ensures ergodicity of the fast process, thus the validity of the averaging principle as we will show next. Besides that, using such parameter does not alter the convergence of the CBO algorithm, as it has recently been shown in \cite[Corollary 3.4]{huang2025faithful}. Moreover, it allows to prove a uniform-in-time mean-field limit \cite{huang2024uniform}, and to construct a self-interacting CBO \cite{huang2025self}.  
    \item The initial condition being deterministic is necessary for our model in order to justify the limit $\varepsilon\to 0$, as we will do in the next section, and in regard to the existing literature on this topic. Nonetheless, it should be noted that, when implemented numerically, the CBO model is able to perform well with random initialization as it is customary done in CBO methods. 
\end{enumerate}
\end{rem}

\subsection{The effective CBO and the averaging principle}\label{sec: averaging}
At the limit $\varepsilon\to 0$, we obtain \textit{averaged (effective)} dynamics, where the dynamics associated to $\mathds{Y}^{\varepsilon}$ disappear and the dynamics of $\mathds{X}^{\varepsilon}$ is averaged with respect to the limiting (ergodic) behavior of $\mathds{Y}^{\varepsilon}$. This yields an effective CBO model which corresponds to the macro-scale dynamics governed by
\begin{equation}\label{ave_model}
    \text{d} \hat{\mathds{X}}_t = \overline{\mathscr{F}} ( \hat{\mathds{X}}_t ) \, \text{d} t + \overline{\mathscr{G}} (\hat{\mathds{X}}_t ) \, \text{d} \mathds{W}^{1}_t,\quad \hat{\mathds{X}}_{0} = \mathbf{x}_0\in \mathbb{R}^{nN}.
\end{equation}
The new \textit{averaged} coefficients are given by 
\begin{equation*}
    \overline{\mathscr{F}}(x)
    := \int \mathscr{F}(x, y) \,  \Upsilon_{x}(\text{d} y) \; \text{ and } \; \overline{\mathscr{G}}(x)
    := \sqrt{\int \mathscr{G} (x, y) \mathscr{G}^{*}(x,y) \,  \Upsilon_{x}(\text{d} y)},
\end{equation*}
the two integrals are defined on $\mathbb{R}^{mMN}\ni y$, where $y$ is a mute variable, and the averaged coefficients are functions of $x\in \mathbb{R}^{nN}$. 
Here, $\Upsilon_{x}$ is the unique invariant probability measure for the dynamics $\mathds{Y}^{x}_{\cdot}$ solution to the \textit{frozen} equation
\begin{equation}\label{eq: frozen eq}
    \text{d} \mathds{Y}^{x}_{t}  =  \mathscr{B}(x,\mathds{Y}^{x}_t)\, \text{d} t +  \mathscr{H}(x,\mathds{Y}^{x}_t)\,\text{d} \mathds{W}^{2}_t,\quad \mathds{Y}^{x}_{0}=\mathbf{y}_0\in \mathbb{R}^{mMN},
\end{equation}
where we have arbitrarily fixed $x\in\mathbb{R}^{nN}$. This corresponds to the dynamics of the fast process $\mathds{Y}$, where we have frozen the slow process $\mathds{X}$ to some value $x$. To avoid any confusion, let us once again explicitly rewrite the dynamics \eqref{ave_model} 
\begin{equation*}
\begin{aligned}
    \text{d} \hat{\mathds{X}}_t = 
    \int \mathscr{F}(\hat{\mathds{X}}_t, y) \,  \Upsilon_{\hat{\mathds{X}}_t}(\text{d} y) \; \text{d} t  + \sqrt{\int \mathscr{G} (\hat{\mathds{X}}_t, y) \mathscr{G}^{*}(\hat{\mathds{X}}_t,y) \, \Upsilon_{\hat{\mathds{X}}_t}(\text{d} y)} \;\, \text{d} \mathds{W}^{1}_t. 
\end{aligned}
\end{equation*}
In other words, the system \eqref{ms_CBO}, whose dimension is $nN\times mMN$, reduces to \eqref{ave_model}, whose dimension is $nN$. 

The convergence of the slow component $\mathds{X}^{\varepsilon}$ in \eqref{ms_CBO} to the process $\hat{\mathds{X}}$ in \eqref{ave_model} has been studied for decades, yet many open questions still persist. In our setting, we will show that the convergence is guaranteed with the recent results in \cite{rockner2021diffusion} (see also \cite[Theorem 2.5]{rockner2019strong}), provided the positive constant $\,\kappa\,$ in  \Cref{sec: drift fast} is small. We shall denote by $C_{b}^{2+\theta}$, with $\theta\in (0,1)$, the set of bounded and $C^{2}$ functions whose second order derivatives are locally H\"older continuous with exponent $\theta$. 

\begin{thm}\label{thm: conv}
Let $T>0$ and $\theta \in (0,1)$. Then for $\,\kappa\in (0,1)$ small, the process $\mathds{X}^{\varepsilon}$ in \eqref{ms_CBO} converges to the process $\hat{\mathds{X}}$ in \eqref{ave_model} in the sense
\begin{equation*}
    \sup _{t \in[0, T]} \left| \mathbb{E}\left[ \varphi(\mathds{X}_{t}^{\varepsilon}) \right] - \mathbb{E}[\varphi(\hat{\mathds{X}}_{t}) ] \right| \leqslant C_{T} \, \varepsilon^{\frac{\theta}{2}}, \quad \forall\, \varphi \in C_{b}^{2+\theta}(\mathbb{R}^{nN}),
\end{equation*}
where $C_{T}>0$ is a constant independent of $\varepsilon$. 
\end{thm}

The proof of the theorem is postponed to Appendix \ref{app: proof conv}. 

\begin{rem}
    We can choose  
    $\theta \in (0,2]$, provided that we smoothen the truncation functions (used in \eqref{ms_CBO}) using mollifiers, as it is customary. In practice, this will not alter the numerical results. Yet we refrain from doing it in the present manuscript as it will make the presentation less readable. 
\end{rem}

\subsection{The algorithm}\label{sec: algo}
We propose an algorithm which approximates the computation of the effective (averaged) dynamics $\hat{\mathds{X}}_{\cdot}$ in \eqref{ave_model}. Without loss of generality, we assume the parameter $R_{1}$ of the truncation functions used in the drift $\mathscr{F}$ and in the diffusion $\mathscr{G}$ is sufficiently large, so that the truncation functions are in fact the identity. This will simplify the computation of the averaged quantities $\overline{\mathscr{F}}$ and $\overline{\mathscr{G}}$. Indeed, in this case, we have $\mathscr{F}(\mathds{X},\mathds{Y}) =  \big(F_{i}(\mathds{X},\mathbf{Y}_{i})\big)_{1\leq i \leq N}$ and $F_{i}(\mathds{X},\mathbf{Y}_{i}) = -\lambda_{1}\,  \left( \mathbf{X}_{i} -  \mathbf{z}( \mathds{X}, \kappa^{-1}\mathbf{v}^{i}) \right)$, 
where 
\begin{equation*}
    \mathbf{z}( \mathds{X}, \kappa^{-1}\mathbf{v}^{i}) = \sum\limits_{r=1}^{N} \mathbf{X}_{r} \frac{\omega_{\alpha}\left(\mathbf{X}_{r},\frac{1}{\kappa}\mathbf{v}^{i}\right)}{\sum\limits_{\ell=1}^{N}\omega_{\alpha}\left(\mathbf{X}_{\ell},\frac{1}{\kappa}\mathbf{v}^{i}\right)} \; \text{ and } \;
    \mathbf{v}^{i} = \sum\limits_{j=1}^{M} Y^{j}_{\mathbf{X}_{i}} \frac{\varpi_{\beta}(\mathbf{X}_{i},Y^{j}_{\mathbf{X}_{i}})}{\sum\limits_{\ell=1}^{M} \varpi_{\beta}(\mathbf{X}_{i},Y^{\ell}_{\mathbf{X}_{i}})}.
\end{equation*}
Its averaged analogue is obtained as $\overline{\mathscr{F}}(\mathds{X}) =  \big( \overline{F}_{i}(\mathds{X})\big)_{1\leq i \leq N}$ and
\begin{equation*}
    \overline{F}_{i}(\mathds{X}) = -\lambda_{1}\,\int  \left( \mathbf{X}_{i} -  \mathbf{z}( \mathds{X}, \kappa^{-1} \mathbf{v}^{i}) \right)\,\text{d}\Upsilon_{\mathds{X}}^{i} = -\lambda_{1} \left( \mathbf{X}_{i} - \int \mathbf{z}(\mathds{X}, \kappa^{-1} \mathbf{v}^{i})\, \text{d}\Upsilon_{\mathds{X}}^{i} \right),
\end{equation*}
where we recall $\mathbf{v}^{i} \equiv \mathbf{v}^{i}(\mathbf{X}_{i},\mathbf{Y}_{i})$, and  $\text{d}\Upsilon_{\mathds{X}}$ defined in the previous section is expressed as a product measure $\text{d}\Upsilon_{\mathds{X}}(y) = \bigotimes\limits_{i=1}^{N}\Upsilon_{\mathds{X}}^{i}(\text{d} y_{i})$ for $y\in \mathbb{R}^{mMN}$ and $y = (y_{i})_{1\leq i \leq N}$. Therefore, the goal is to estimate the \textit{averaged consensus} point given 
\begin{equation}\label{eq: averaged z}
    \overline{\mathbf{z}}_{i}(\mathds{X}) := \int_{\mathbb{R}^{mM}} \mathbf{z}(\mathds{X}, \kappa^{-1} \mathbf{v}^{i}(y_{i}))\, \Upsilon_{\mathds{X}}^{i}(\text{d} y_{i})
\end{equation}
for any given fixed vector $\mathds{X}$. 
This is done using the relationship  between the invariant measure and the averaged long-time behavior:
\begin{equation}\label{eq: time average}
    \int_{\mathbb{R}^{mM}} \mathbf{z}(\mathds{X}, \kappa^{-1} \mathbf{v}^{i}(\mathbf{X}_{i},y_{i}))\, \Upsilon_{\mathds{X}}^{i}(\text{d} y_{i}) = \lim\limits_{T\to +\infty} \frac{1}{T}\int_{0}^{T}  \mathbf{z}(\mathds{X},  \kappa^{-1} \mathbf{v}^{i}(\mathbf{X}_{i},\mathbf{Y}_{i,t}))\, \text{d} t,
\end{equation}
where $\mathbf{Y}_{i}$ solves \eqref{eq: frozen eq} with $x = \mathds{X} = (\mathbf{X}_{i})_{1\leq i \leq N}$ is a fixed parameter therein (constant in time). See for example \cite[equation (6.4.2)]{pavliotis2008multiscale}.

For the diffusion term, instead of computing $\overline{\mathscr{G}}$ as defined in the previous subsection, we shall approximate it with a diagonal matrix whose elements are the components of the vector
\begin{equation*}
    \sigma_{1}\left( \delta + \big| \mathbf{X}_{i} - \overline{\mathbf{z}}_{i}(\mathds{X})\big| \right).
\end{equation*}
Here, we use the notation $|v| = (|v_{i}|)_{1\leq i \leq d}$, for any vector $v\in \mathbb{R}^{d}$, and the constant $\delta$ is a vector of the same dimension whose elements are all equal to $\delta$.

Thus, at each iteration of the algorithm, we would like to retrieve a quantity that becomes closer to $\overline{\mathbf{z}}_{i}(\mathds{X})$ as defined in \eqref{eq: averaged z}, using the averaging in time as in the right-hand side of \eqref{eq: time average}. We proceed in a similar fashion as \cite[Algorithm 1]{chaudhari2019entropy}, using a Forward--Looking Averaging (instead of a linear averaging), that is: 
\begin{enumerate}[label=\textbf{FLA:},ref=FLA]
    \item \label{FLA} we compute $\overline{\mathbf{z}}_{i}(\mathds{X})$ at time step $t+1$  as a convex combination of its value at the previous step $t$, and the value of $\mathbf{z}( \mathds{X}, \kappa^{-1}\mathbf{v}^{i})$ at time step $t+1$; see the forward looking averaging (FLA) step in \texttt{Algorithm \ref{algo averaging}} below.
\end{enumerate}
The FLA method is comparable to the exponentially weighted moving average \cite[process $Z_t$ in page 510]{yeh2003multivariate}, which is an example of the more general class of \textit{moving average} methods as known in statistics.

\begin{rem}
Note that the computation of the invariant measure could be improved and/or realized differently, for example using the results in \cite{alrachid2019new, chassagneux2024computing}. In \cite{alrachid2019new}, the invariant measure is computed using a population of interacting particles. Using such approximation in our present setting would require the introduction of new interacting particles which would make the model less readable. In \cite{chassagneux2024computing}, the invariant measure is computed using a self-interacting particle. This approach is reminiscent to  \eqref{FLA}, since both of them use the past information of the particle, but compute the average with different weights. Ultimately, we have chosen \eqref{FLA} as it computes directly the desired averaged quantity, rather than the invariant measure (as in the mentioned references). 
\end{rem}

In the following, we present the pseudo-code for multiscale CBO solving the bi-level optimization problem (\ref{bilevel opt}).

\RestyleAlgo{ruled}
\SetKwComment{Comment}{/* }{ */}
\begin{algorithm}[htb]
\textbf{Set parameters:} $\gamma, \, \alpha, \, \beta, \, \lambda_1, \, \lambda_2, \, \sigma_1, \, \sigma_2, \, \Delta t, \, \Delta \tau, \, T_{x}, \, T_{y},\,R_1,\,R_2,\,\delta_1,\,\delta_2\,,\kappa\;$\;
\textbf{Initialize:} $\mathbf{z}^{i}_{0}=X_0^i \sim \rho_0, \;  i=1, \ldots, N\;$\;
\qquad\qquad\quad $Y^j_{i,0} \sim \mu_0, \; j=1, \ldots, M$, where $M\ll N$\;
\qquad\qquad\quad$t=0$, $\mathcal{T}_{x}=T_{x}/\Delta t$, $\mathcal{T}_y=T_y/\Delta \tau$\;
\While {$t\leq \mathcal{T}_{x}$ }{
\For{$i=1, \ldots, N$}{
$\tau=0$\;
$\mathbf{v}^i_{\tau} = \sum\limits_{j=1}^{M} Y^{j}_{X_{i},\tau} \frac{\varpi_{\beta}(X^{i}_{t}, Y^{j}_{X_{i},\tau})}{\sum\limits_{\ell=1}^{M}\varpi_{\beta}(X^{i}_{t},Y^{\ell}_{X_{i},\tau})}$ \Comment*[r]{Initial consensus of $\mathbf{Y}_{i,0}$}
\While{$\tau\leq \mathcal{T}_y$}{
\For{$j=1, \ldots, M$}{
\text{sample}  $W_{X_{i},\tau}^{Y,j}$ from $\mathcal{N}(0,\Delta \tau \, \mathds{I}_{m})$\;
$Y^{j}_{X_{i},\tau} \gets Y^{j}_{X_{i},\tau}- \lambda_2 \psi_{R_2}\!\left(Y^{j}_{X_{i},\tau}- \kappa\,\mathbf{v}^i_{\tau}\right)\Delta\tau + \sigma_2 D_{R_2}^{\delta_2}\!\left(Y^{j}_{X_{i},\tau}- \kappa\,\mathbf{v}^i_{\tau}\right) W_{X_{i},\tau}^{Y,j}$\;
}
$\mathbf{v}^i_{\tau} \gets \sum\limits_{j=1}^{M} Y^{j}_{X_{i},\tau} \frac{\varpi_{\beta}(X^{i}_{t}, Y^{j}_{X_{i},\tau})}{\sum\limits_{\ell=1}^{M}\varpi_{\beta}(X^{i}_{t},Y^{\ell}_{X_{i},\tau})}$ \Comment*[r]{consensus of $\mathbf{Y}_{i,\tau}$}
$\mathbf{z}' \gets \sum\limits_{k=1}^{N} X^{k}_{t} \frac{\omega_{\alpha}(X^{k}_{t},\frac{1}{\kappa}\mathbf{v}^i_{\tau})}{\sum\limits_{r=1}^{N}\omega_{\alpha}(X^{r}_{t},\frac{1}{\kappa}\mathbf{v}^i_{\tau})}$ \Comment*[r]{auxiliary variable}
$\mathbf{z}^{i}_{t}\gets (1-\gamma) \mathbf{z}^{i}_{t} + \gamma \mathbf{z}'$ \Comment*[r]{\ref{FLA} of $X^{i}_{t}$; depends on $i$}
$\tau \gets \tau +1$\;
}
\text{sample}  $W_{i,t}^{X}$ from $\mathcal{N}(0,\Delta t \,\mathds{I}_{n})$\;
$X^{i}_{t} \gets X^{i}_{t}-\lambda_{1}\psi_{R_1}\left(X_{t}^i-\mathbf{z}^{i}_{t}\right)\Delta t + \sigma_{1}D_{R_1}^{\delta_1}\left( X_{t}^i - \mathbf{z}^{i}_{t}\right)W_{i,t}^{X}$ \;
}
$t \gets t+1$\;
}
$(\bar{X}, \, \bar{\mathbf{v}}) \gets \left( \frac{1}{N}\sum_{i=1}^N X_{\mathcal{T}_x}^i ,\; \frac{1}{N}\sum_{i=1}^N \mathbf{v}_{\mathcal{T}_y}^i\right)$\;
$X^* \gets \sum\limits_{r=1}^{N} X_{\mathcal{T}_x}^r \frac{\omega_{\alpha}\left(X_{\mathcal{T}_x}^r, \frac{1}{\kappa}\bar{\mathbf{v}}\right)}{\sum\limits_{\ell=1}^{N}\omega_{\alpha}\left(X_{\mathcal{T}_x}^\ell, \frac{1}{\kappa}\bar{\mathbf{v}}\right)}$\;
$Y^* \gets \frac{1}{\kappa}\sum\limits_{r=1}^{N} \mathbf{v}_{\mathcal{T}_y}^r \frac{\varpi_{\beta}\left(\bar{X}, \, \frac{1}{\kappa}\mathbf{v}_{\mathcal{T}_y}^r\right)}{\sum\limits_{\ell=1}^{N}\varpi_{\beta}\left(\bar{X},\, \frac{1}{\kappa}\mathbf{v}_{\mathcal{T}_y}^\ell \right)}$\;
\textbf{Return} $(X^*, Y^*)$ \Comment*[r]{Approximation of solution to \eqref{bilevel opt}}
\caption{Multiscale CBO for bi-level optimization problems}
\label{algo averaging}
\end{algorithm}

In \texttt{Algorithm \ref{algo averaging}}, the weight function are defined according to the objective functions of upper and lower level optimization problems (see \eqref{weight beta}-\eqref{weight alpha}), i.e.,
\begin{equation*}
    \varpi_{\beta}(x,y) = \exp(-\beta G(x,y)) \quad \text{and} \quad \omega_{\alpha}(x,y) = \exp(-\alpha F(x,y)).
\end{equation*}

\begin{rem}
    It is important to note that for running \texttt{Algorithm \ref{algo averaging}}, we do not need to have as many $y$-particles as $x$-particles, in particular we have $M\ll N$. Moreover, it is not necessary to run the dynamics of the $y$-particles for a long time, that is, an approximate value of their consensus is already enough for the algorithm to perform well. Hence, we can choose $T_{y}\ll T_{x}$. 
    This can be justified by the fact that \textit{errors} in the computation of the consensus for the $y$-particles, can be tamed during the averaging procedure used for computing  the consensus of the $x$-particles. On the other hand, the \texttt{for} loop  can be computed in parallel, since the groups of particles $\mathbf{Y}_i=\{Y_{\mathbf{X}_{i}}^1,\cdots,Y_{\mathbf{X}_{i}}^M\}$  are independent from $\mathbf{Y}_k$ for any $k\neq i$.
    \label{rem: model reduction}
\end{rem}

\begin{rem}
\label{paramater_c}
    Although Theorem \ref{thm: conv} establishes the convergence of the model \eqref{ms_CBO} to an averaged dynamics as $\varepsilon\rightarrow 0$   under the assumption $\kappa\in (0,1)$ small,  numerical results suggest that this smallness assumption is not necessary for the convergence to hold. Indeed, in practice (see \texttt{Algorithm \ref{algo averaging}}), we implement an approximation of the averaging with respect to the invariant measure (see the discussion at the beginning of \Cref{sec: algo}). 
    In other words, we do not implement the \textit{effective} (averaged) dynamics which is the one obtained at the limit $\varepsilon\to 0$, because this would require an explicit formula for the invariant measure which is difficult to construct (analytically). To bypass this difficulty, what we implement instead is an approximation of this limiting behavior. This approximation is obtained with a finitely--small $\varepsilon>0$. Hence, there is no need to restrict the choice of the parameter $\kappa$ in practice.
\end{rem}

\section{Multiscale CBO for multi-level optimization}\label{sec: generalize}

\subsection{The general strategy}
\label{sec:multi}
The bi-level approach presented in the previous sections can be extended to tackle multi-level optimization problems of the form
\begin{equation}\label{multilevel opt}
\begin{aligned}
    & \min\limits_{x_{0}\in \mathbb{R}^{d_{0}}} \; F_{0}(x_{0},y_{1}, \dots, y_{n})\\
    &\quad \quad  \text{ s.t. } \; y_{1} \in \argmin\limits_{y\in \mathbb{R}^{d_{1}}}  F_{1}(x_{0},y,y_{2},\dots, y_{n}) \\
    &\quad \quad \quad \quad  \text{ s.t. }  \; y_{2} \in \argmin\limits_{y\in \mathbb{R}^{d_{2}}}  \; F_{2}(x_{0},y_{1},y, y_{3},\dots, y_{n}) \\
    &\quad \quad \quad \quad \quad \quad \quad \ddots \quad \quad \quad  \quad \ddots  \\
    &\quad \quad \quad \quad \quad \quad \quad \quad  \text{ s.t. } \; y_{n} \in \argmin\limits_{y\in \mathbb{R}^{d_{n}}}   \; F_{n}(x_{0},\dots, y_{n-1},y) 
\end{aligned}
\end{equation} 
In this case, the multiscale CBO model would have as many scales as the number of levels in \eqref{multilevel opt}. Its system of interacting SDEs is analogous to  \eqref{ms_CBO}, that is 
\begin{equation}
\left\{\;
\begin{aligned}
    &\text{d} \mathds{X}_{t} =  \mathscr{F}(\mathds{X}_t,\mathds{Y}^{1}_t, \dots, \mathds{Y}^{n}_t)\,\text{d} t + \mathscr{G}(\mathds{X}_t,\mathds{Y}^{1}_t, \dots, \mathds{Y}^{n}_t)\,\text{d} \mathds{W}^{0}_{t}\\
    &\text{d} \mathds{Y}^{1}_{t} = \frac{1}{\varepsilon}\mathscr{B}_{1}(\mathds{X}_t,\mathds{Y}^{1}_t, \dots, \mathds{Y}^{n}_t)\, \text{d} t + \frac{1}{\sqrt{\varepsilon}} \mathscr{H}_{1}  (\mathds{X}_t,\mathds{Y}^{1}_t, \dots, \mathds{Y}^{n}_t)\,\text{d} \mathds{W}^{1}_t\\
    & \quad \quad \; \vdots \\
    &\text{d} \mathds{Y}^{n}_{t} = \frac{1}{\varepsilon^{n}}\mathscr{B}_{n}(\mathds{X}_t,\mathds{Y}^{1}_t, \dots, \mathds{Y}^{n}_t)\, \text{d} t + \frac{1}{\sqrt{\varepsilon^{n}}} \mathscr{H}_{n}  (\mathds{X}_t,\mathds{Y}^{1}_t, \dots, \mathds{Y}^{n}_t)\,\text{d} \mathds{W}^{n}_t.
\end{aligned}
\right.
\end{equation}
The analysis can be performed in a cascade manner: First we consider $\mathds{Y}^{n}_{\cdot}$ as being the only \textit{fast} variable (while $\mathds{X}_{\cdot},\mathds{Y}^{1}_{\cdot}, \dots, \mathds{Y}^{n-1}_{\cdot}$ are frozen), this yields an averaged system made of the interacting SDEs $\hat{\mathds{X}}_{\cdot},\hat{\mathds{Y}}^{1}_{\cdot}, \dots, \hat{\mathds{Y}}^{n-1}_{\cdot}$ (these are analogously defined as in \eqref{ave_model}). Then we repeat the process considering $\hat{\mathds{Y}}^{n-1}_{\cdot}$ as being the only \textit{fast} variable (while $\hat{\mathds{X}}_{\cdot},\hat{\mathds{Y}}^{1}_{\cdot}, \dots, \hat{\mathds{Y}}^{n-2}_{\cdot}$ are frozen), and so on. 

\subsection{Tri-level optimization problems}

We consider the tri-level optimization problem (\ref{trilevel opt}) as an example of the multi-level problem. 
To generalize the algorithm in 
\Cref{sec: implement},  we consider $N$ particles $\{X^{i}_{\cdot}\}_{i=1}^{N}$ in $\mathbb{R}^{n}$, where $n,N$ are fixed integers. To each particle $X^{i}_{\cdot}\in \mathbb{R}^{n}$ with $1 \leq i \leq N$, we associate a set of $M$ particles $(Y^{1}_{X^{i},\cdot}, \dots, Y^{M}_{X^{i},\cdot}) \in \left(\mathbb{R}^{m}\right)^{M}$, where $Y_{X^{i},\cdot}^{j}\in\mathbb{R}^m$ for any $1\leq j \leq M$. 
Then, for every particle $Y^j_{X^{i},\cdot}\in \mathbb{R}^{m}$ with $1 \leq i \leq N$ and $1 \leq j \leq M$, we associate a set of $P$ particles $(R^{1}_{(i,j),\cdot}, \dots, R^{P}_{(i,j),\cdot}) \in \left(\mathbb{R}^{p}\right)^{P}$, where $R^k_{(i,j),\cdot}\in\mathbb{R}^p$ for any $1\leq k \leq P$. The subscript $(i,j)$ refers to the dependency on both $X^{i}$ and $Y^{j}_{X^{i}}$.
This corresponds to \texttt{Approach 1} in \Cref{A_1}.
\begin{figure}[h]
\centering
\begin{subfigure}[b]
  {0.52\textwidth}
\includegraphics[width=1.05\textwidth]
{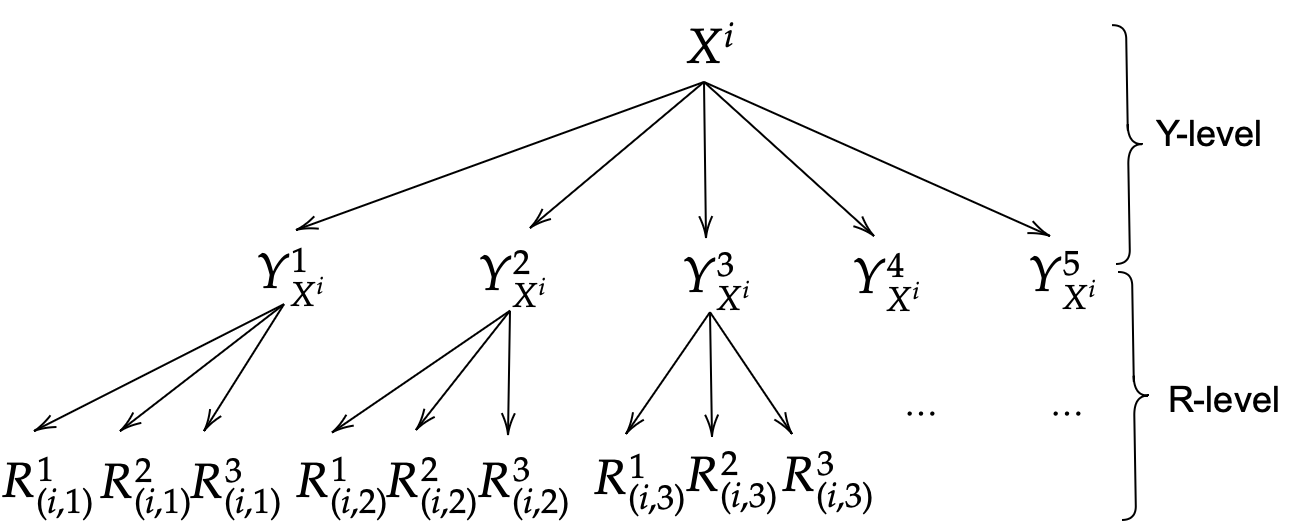}
\caption{\texttt{Approach 1}}
\end{subfigure}\quad\quad
\begin{subfigure}[b]
  {0.42\textwidth}
\includegraphics[width=1.05\textwidth]
{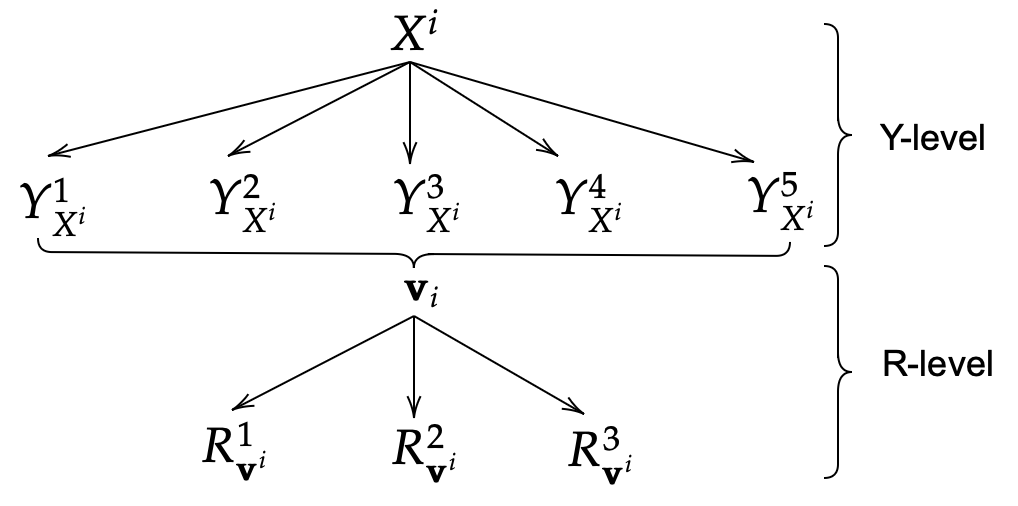}   
 \caption{\texttt{Approach 2}}
\end{subfigure}
\caption{Particle systems in tri-level optimization algorithm with $M=5$ particles for $Y$-system and $P=3$ particles for $R$-system. \texttt{Y-level} corresponds to the $Y$-particles solving the second level of optimization. In both approaches, they correspond to each fixed $X$-particle (here $X^{i}$). \texttt{R-level} corresponds to the $R$-particles solving the third level of optimization. In \texttt{Approach 1}, they correspond to each fixed $Y$-particle, whereas in \texttt{Approach 2}, they correspond to a representative (consensus) of the $Y$-particles.}
\label{A_1}
\end{figure}

However, in practice, this straightforward  generalization is computationally expensive, since we  need $N$ particles in $X$-system, $N\times M$ particles in $Y$-system and $N\times M\times P$ particles in $R$-system. Therefore, for tri-level optimization problems, we propose \texttt{Approach 2} as illustrated in \Cref{A_1}. In the latter, we shall consider $N$ particles $\{ X^{i}_{\cdot}\}_{i=1}^{N}$ in $\mathbb{R}^{n}$, and  $M$ particles $(Y^{1}_{X^{i},\cdot}, \dots, Y^{M}_{X^{i},\cdot}) \in \left(\mathbb{R}^{m}\right)^{M}$ associated to each particle $X^{i}_{\cdot}$ as for bi-level problems. Then, 
we compute the consensus point $\mathbf{v}^i$ of $\{Y^{j}_{X^{i},\cdot}\}_{j=1}^M$, for $1\leq i\leq N$, to which we associate a set of $P$ particles $\{R^{k}_{\mathbf{v}^i,\cdot}\}_{k=1}^{P}\in \left(\mathbb{R}^{p}\right)^{P}$. 
The dynamics now only have $N$ particles in $X$-system, $N\times M$ particles in $Y$-system and $N\times P$ particles in $R$-system. 

The pseudo-code corresponding to \texttt{Approach 2} in \Cref{A_1} is \texttt{Algorithm \ref{algo: tri-level}}, where the weight functions $\omega_{i}=\omega_{i}(x,y,r),i=1,2,3$, for some $\alpha_{1}, \alpha_{2}, \alpha_{3}>0$  are
\begin{equation*}
    \omega_{1} = \exp(-\alpha_1 F(x,y,r)),\; \omega_{2} = \exp(-\alpha_2 G(x,y,r)),\; \omega_{3} = \exp(-\alpha_3 E(x,y,r)).
\end{equation*}

Similar to \texttt{Algorithm \ref{algo averaging}} for bi-level problems, we can choose $M\ll N$, and $P\ll N$. Additionally, we set  
for simplicity $\kappa=1$. We also introduce the notation $\mathbf{D}\equiv \mathbb{D}_Q^\delta$ and $\Psi\equiv \psi_Q$, which are the truncation functions defined in \S\ref{sec:modified CBO}. Here, the truncation parameter (denoted therein by $R$) is now a fixed constant $Q>0$.  To gain more clarity in \texttt{Algorithm \ref{algo: tri-level}}, for $1\leq i \leq N$ and $t\geq 0$ we shall use the notation 
\begin{equation}\label{notation tri}
\begin{aligned}
    \mathbf{Y}_{i,t} \, := \; & \mathbf{Y}_{\mathbf{X}^{i},t} = (Y^{1}_{\mathbf{X}^{i},t},\dots,Y^{M}_{\mathbf{X}^{i},t}) \; \in (\mathbb{R}^{m})^M ,\\
    \mathbf{R}_{i,t} \, := \; & \mathbf{R}_{\mathbf{v}^i,t} = (R^{1}_{\mathbf{v}^i,t},\dots,R^{P}_{\mathbf{v}^i,t})\; \in (\mathbb{R}^{p})^P ,\\
    &  \text{ where } \;  \mathbf{v}^i \; \text{ is the consensus of } \; \{Y^{j}_{\mathbf{X}^{i},\cdot}\}_{j=1}^M \, \text{ in } \eqref{consensus v i}.
\end{aligned}
\end{equation}
As described in \S\ref{sec:multi}, we proceed in a cascade manner: We focus first on the lowest level (the $R$-particles) while keeping the other ($X,Y$)-particles fixed. We run the standard CBO algorithm for the $R$-particles. Then we move to the next higher level (the $Y$-particles). Here we must take the averaged consensus of the $Y$-particles with respect to the already computed consensus of the $R$-particles. At this stage, the highest level ($X$-particles) is still fixed, while the consensus term in the $Y$-particles is averaged using the \eqref{FLA} method so that it includes the effect of the  $R$-particles. Thus we are left with two populations: the $X$-particles which have remained fixed so far, and the $Y$-particles whose consensus terms have been averaged according to the (now, disappeared) $R$-particles. This concludes the first \textit{averaging} method, and reduces the setting from a $3$-scale to $2$-scale dynamics. We can then continue as in \texttt{Algorithm \ref{algo averaging}.} This procedure is highlighted in \texttt{Algorithm \ref{algo: tri-level words}} where the green lines correspond to what is different from \texttt{Algorithm \ref{algo averaging}}. This is  expressed with symbols in \texttt{Algorithm \ref{algo: tri-level}}.

\begin{algorithm}[H]
 \textbf{Set parameters:} $\gamma, \, \alpha_1,\, \alpha_2,\, \alpha_3, \, \lambda,\, \sigma, \, \Delta t, \, T_{x}, \, T_{y},\, T_{r},\,\delta$,\, $\kappa=1$ \;
 \textbf{Initialize:} $\mathbf{z}^{i}_{0}=X_0^i \sim \rho_0, \;  i=1, \ldots, N$;  $Y^j_{i,0} \sim \mu_0, \; j=1, \ldots, M$;
 $R^k_{i,0} \sim \gamma_0, \; k=1, \ldots, P$; 
$t=0$, $\mathcal{T}_{x}=T_{x}/\Delta t$, $\mathcal{T}_y=T_y/\Delta t$, $\mathcal{T}_r=T_r/\Delta t$ \;
\While{$t\leq \mathcal{T}_{x}$}{
\For {$i=1, \ldots, N$}{
Initialize time index $(\tau=0)$ for the $Y$-particles\;
 Initialize consensus point $\mathbf{v}_{0}^i$ of $\mathbf{Y}_{i,0}$ \;
\While{ $\tau\leq \mathcal{T}_y$}{
 \begin{thisnote}
 Initialize time index $(s=0)$ for the $R$-particles\;
 Initialize consensus point $\mathbf{r}_{0}^i$ of $\mathbf{R}_{i,0}$ \;
\While{ $s\leq \mathcal{T}_r$}{ 
Run CBO for the $R$-particles\;
Compute consensus point $\mathbf{r}_{s}^i$ of of $\mathbf{R}_{i,s}$\;
Update the \textit{averaged} consensus of the $Y$-particles (using \eqref{FLA})\;
Update time index of $R$-particles ($s \gets s+1$)\;
}
 \end{thisnote}
(Note that $s = \mathcal{T}_{r}$ here)\;
Run CBO for the $Y$-particles\;
Compute consensus point $\mathbf{v}_{s}^i$ of  $\mathbf{Y}_{i,s}$\;
Update the \textit{averaged} consensus of the $X$-particles (using \eqref{FLA})\;
Update time index of $Y$-particles ($\tau \gets \tau+1$)\;
}
}
Run CBO for the $X$-particles\;
Update time index of $X$-particles ($t \gets t+1$)\;
}
Compute the arithmetic mean of $(X_{\mathcal{T}_x}^i)_{i=1}^{N}$, $(\mathbf{v}_{\mathcal{T}_y}^i)_{i=1}^{N}$, and $(\mathbf{r}_{\mathcal{T}_r}^i)_{i=1}^{N}$\;
(The latter are needed to compute the weights in the consensus next)\;
Compute the consensus points of $(X^*, Y^*, R^*)$\;
\textbf{Return} $(X^{*}, Y^{*}, R^{*})$  \Comment*[r]{Approximation of solution to \eqref{trilevel opt}}
\caption{Multiscale CBO for tri-level optimization -- in words.}
\label{algo: tri-level words}
\end{algorithm}

Recalling the notation in \eqref{notation tri}, the algorithm for tri-level optimization problem can be expressed as follows. 

\begin{algorithm}[H]
 \textbf{Set parameters:} $\gamma, \, \alpha_1,\, \alpha_2,\, \alpha_3, \, \lambda,\, \sigma, \, \Delta t, \, T_{x}, \, T_{y},\, T_{r},\,\delta$,\, $\kappa=1$ \;
 \textbf{Initialize:} $\mathbf{z}^{i}_{0}=X_0^i \sim \rho_0, \;  i=1, \ldots, N$;  $Y^j_{i,0} \sim \mu_0, \; j=1, \ldots, M$;
 $R^k_{i,0} \sim \gamma_0, \; k=1, \ldots, P$; 
$t=0$, $\mathcal{T}_{x}=T_{x}/\Delta t$, $\mathcal{T}_y=T_y/\Delta t$, $\mathcal{T}_r=T_r/\Delta t$ \;
\While{$t\leq \mathcal{T}_{x}$}{
\For {$i=1, \ldots, N$}{
$\tau=0$ \;
$\mathbf{v}_{\tau}^i = \sum\limits_{j=1}^{M} Y^{j}_{i,\tau} \frac{\omega_{2}(X_t^{i}, Y^{j}_{i,\tau},\frac{1}{P}\sum_{k=1}^P R_{i,0}^k)}{\sum\limits_{\ell=1}^{M}\omega_{2}(X_t^{i},Y^{\ell}_{i,\tau},\frac{1}{P}\sum_{k=1}^P R_{i,0}^k)}$ \Comment*[r]{Consensus of $\mathbf{Y}_{i,0}$}
\While{ $\tau\leq \mathcal{T}_y$}{
 \begin{thisnote}
 $s=0$ \;
$\mathbf{r}^i_{s} \gets \sum\limits_{k=1}^{P} R^{k}_{i,s} \frac{\omega_{3}(X_t^{i}, \mathbf{v}^i_{\tau},R^{k}_{i,s})}{\sum\limits_{\ell=1}^{M}\omega_{3}(X^{i}_t,\mathbf{v}^M_{i,\tau},R^{\ell}_{i,s})}$ \Comment*[r]{Consensus of $\mathbf{R}_{i,0}$}
\While{ $s\leq \mathcal{T}_r$}{
$\mathbf{R}_{i,s} \gets \mathbf{R}_{i,s}- \lambda\;\Psi\left(\mathbf{R}_{i,s}-\mathbf{r}^i_{s}\right)\Delta t+\sigma \mathbf{D}\left(\mathbf{R}_{i,s}-\mathbf{r}^i_{s}\right)W_{i,s}^{\mathbf{R}}$\;
$\mathbf{r}^i_{s} \gets \sum\limits_{k=1}^{P} R^{k}_{i,s} \frac{\omega_{3}(X_t^{i}, \mathbf{v}^i_{\tau},R^{k}_{i,s})}{\sum\limits_{\ell=1}^{P}\omega_{3}(X^{i}_t,\mathbf{v}^i_{\tau},R^{\ell}_{i,s})}$ \Comment*[r]{consensus of $\mathbf{R}_{i,s}$}
$\mathbf{v}' \gets \sum\limits_{j=1}^{M} Y^{j}_{i,\tau} \frac{\omega_{2}(X^{i}_t, Y^{j}_{i,\tau},\mathbf{r}_{s}^i)}{\sum\limits_{\ell=1}^{M}\omega_{2}(X^{i}_t,Y^{\ell}_{i,\tau},\mathbf{r}_{s}^i)}$ \;
 $\mathbf{v}^{i}_\tau \gets (1-\gamma) \mathbf{v}^{i}_\tau + \gamma \mathbf{v}'$ \Comment*[r]{\ref{FLA} of $\mathbf{Y}_{i,\tau}$}
$s \gets s+1$\;
}
 \end{thisnote}
 $\mathbf{Y}_{i,\tau} \gets \mathbf{Y}_{i,\tau}- \lambda\;\Psi\left(\mathbf{Y}_{i,\tau}-\mathbf{v}^{i}_\tau\right)\Delta t+\sigma \mathbf{D}\left(\mathbf{Y}_{i,\tau}-\mathbf{v}^{i}_\tau\right)W_{i,\tau}^{\mathbf{Y}}$ \;
$\mathbf{v}^i_{\tau} = \sum\limits_{j=1}^{M} Y^{j}_{i,\tau} \frac{\omega_{2}(X_t^{i}, Y^{j}_{i,\tau},\mathbf{r}_{s}^i)}{\sum\limits_{\ell=1}^{M}\omega_{2}(X_t^{i},Y^{\ell}_{i,\tau},\mathbf{r}_{s}^i)}$ \Comment*[r]{ Note that $s = \mathcal{T}_{r}$}
$\mathbf{z}' \gets \sum\limits_{q=1}^{N} X^{q}_t \frac{\omega_{1}(X^{q}_t,\mathbf{v}_{\tau}^i,\mathbf{r}_{s}^i)}{\sum\limits_{r=1}^{N}\omega_{1}(X_t^{r},\mathbf{v}_{\tau}^i,\mathbf{r}_{s}^i)}$ \;
$\mathbf{z}_t^{i}\gets (1-\gamma) \mathbf{z}_t^{i} + \gamma \mathbf{z}'$ \Comment*[r]{\ref{FLA} of $X^{i}_t$; it depends on $i$}
$\tau \gets \tau+1$
}
}
$\mathbf{X}^{i}_t \gets \mathbf{X}^{i}_t-\lambda\Psi\left(\mathbf{X}^i_t-\mathbf{z}^{i}_t\right)\Delta t + \sigma \mathbf{D}\left(\mathbf{X}^i_t - \mathbf{z}_t^{i}\right)W^{\mathbf{X}}_{i,t}$ \;
$t \gets t+1$ \;
}
$(\bar{X},\, \bar{\mathbf{v}},\, \bar{\mathbf{r}}) \gets \left(  \frac{1}{N}\sum_{i=1}^N X_{\mathcal{T}_x}^i,\; \frac{1}{N}\sum_{i=1}^N \mathbf{v}_{\mathcal{T}_y}^i,\; \frac{1}{N}\sum_{i=1}^N \mathbf{r}_{\mathcal{T}_r}^i\right)$\;
$X^* \gets \sum\limits_{k=1}^{N} X_{\mathcal{T}_x}^k\frac{\omega_{1}(X_{\mathcal{T}_x}^k, \bar{\mathbf{v}}, \bar{\mathbf{r}})}{\sum\limits_{\ell=1}^{N}\omega_{1}(X_{\mathcal{T}_x}^\ell,\bar{\mathbf{v}},\bar{\mathbf{r}})}$\;
$Y^* \gets \sum\limits_{k=1}^{N} \mathbf{v}_{\mathcal{T}_y}^k\frac{\omega_{2}(\bar{X}, \mathbf{v}_{\mathcal{T}_y}^k,\bar{\mathbf{r}})}{\sum\limits_{\ell=1}^{N}\omega_{2}(\bar{X},\mathbf{v}_{\mathcal{T}_y}^\ell,\bar{\mathbf{r}})}$\;
$R^* \gets \sum\limits_{k=1}^{N} \mathbf{r}_{\mathcal{T}_r}^k\frac{\omega_{3}(\bar{X},  \bar{\mathbf{v}},\mathbf{r}_{\mathcal{T}_r}^k)}{\sum\limits_{\ell=1}^{N}\omega_{3}(\bar{X},\bar{\mathbf{v}},\mathbf{r}_{\mathcal{T}_r}^\ell)}$\;
\textbf{Return} $(X^{*}, Y^{*}, R^{*})$  \Comment*[r]{Approximation of solution to \eqref{trilevel opt}}
\caption{Multiscale CBO for tri-level optimization.}
\label{algo: tri-level}
\end{algorithm}

\section{Numerical experiments}\label{sec: test}

In this section, we first assess the performance of the multiscale CBO algorithm for bi-level and tri-level optimization problems. Then, we compare the performance of multiscale CBO methods with the CBO methods proposed by  \cite{huang2024consensus}  in solving min-max optimization problems. Lastly, we conduct a sensitivity analysis of the \texttt{MS-CBO} parameters specifically for bi-level optimization problems. We shall construct multi-level optimization problems using the following  benchmark functions:
\begin{enumerate}[label=\textcolor{blue}{\textbf{(\arabic*)}}]
    \item  \textbf{Ackley function \cite{Ackley}:}
 \begin{equation*}
        A(x):= -20 \exp \left(-0.2 \sqrt{\frac{1}{d} \sum_{i=1}^{d} x_i^2}\right)-\exp \left(\frac{1}{d} \sum_{i=1}^{d} \cos \left(2\pi x_i\right)\right)+\exp(1)+20.
    \end{equation*}
     \item  \textbf{Rastrigin function \cite{rastrigin1974systems}:}
    \begin{equation*}
        R(x):=\sum_{i=1}^{d} x_i^2+1.5(1- \cos \left(2 \pi x_i\right)).
    \end{equation*}
     \item  \textbf{Levy function:}
    \begin{equation*}
        \begin{aligned}
            L(x):=
            & \sin ^2\left(\pi \left(1+\frac{x_1}{4}\right)\right)+\sum_{i=1}^{d-1}\left(\frac{x_i}{4}\right)^2\left[1+10 \sin ^2\left(\pi \left(1+\frac{x_i}{4}\right)
            +1\right)\right]\\
            &\quad \quad \quad  +\left(\frac{x_d}{4}
            \right)^2\left[1+\sin ^2\left(2 \pi \left(1+\frac{x_i}{4}\right)
            \right)\right],\qquad d\geq2.
        \end{aligned}
    \end{equation*}
    \end{enumerate}
All the benchmarks have global minimizer $x^* =(0,\cdots,0)$. In the following, we perform 100 Monte Carlo simulations to compute the success rate of the multiscale CBO algorithm and expectation of the error, 
which is defined as
\[
\texttt{error} := \|X^*-x^*\|_2+\|Y^*-y^*\|_2
\]
with output $(X^*,Y^*)$ for bi-level and min-max optimization algorithms, $(x^*,y^*)$ is the true solution. For tri-level optimization problem,  given the output $(X^*,Y^*,R^*)$ of \texttt{Algorithm \ref{algo: tri-level}}, the error is defined analogously as 
\[
\texttt{error} := \|X^*-x^*\|_2+\|Y^*-y^*\|_2+\|R^*-r^*\|_2,
\]
where $(x^*,y^*,r^*)$ is the true solution of the tri-level optimization problem. We consider a run is successful if 
\begin{equation}\label{test}
    \texttt{error}\leq 0.25,
\end{equation}
following what has previously been done in \cite{kalise2023consensus,pinnau2017consensus}. All numerical experiments were implemented using \textsc{Matlab R2024b} on a MacBook Pro equipped with an Apple M4 chip and 16 GB of RAM.  \textsc{Matlab} codes for \texttt{MS-CBO} and the numerical tests presented in this paper are available in the GitHub repository: \\
\url{https://github.com/AmberYuyangHuang/Multi-scale-CBO}

\subsection{Bi-level optimization}
\label{sec:bilevel}
In this section, we consider the bi-level optimization problem (\ref{bilevel opt}). The particle dynamics are discretized using the Euler–Maruyama scheme. The parameters are chosen as
\begin{equation}
\label{paramaters}
\begin{aligned}
    &\alpha=\beta=10^{15},\quad T_x=50, \quad T_y=0.5,\quad \Delta t=\Delta \tau=0.1\\
    & N=100,\; M=25,\quad\lambda_1=\lambda_2=1 ,\quad \sigma_1=\sigma_2=2,\quad \gamma=0.75,\\
    & \delta_1=\delta_2=10^{-5}, \quad R_1=R_2=10, \quad \kappa=1, \quad n=m=d
\end{aligned}
\end{equation}
The initial positions of $x$-particles and  $y$-particles are sampled from  $\rho_{0}=\mathcal{U}[-1,3]^d$ and $\mu_{0}=\mathcal{U}[-1,3]^d$,  respectively.  

\begin{rem}
    Let us comment on the values in \eqref{paramaters} for the choice of parameters. 
    \begin{enumerate}[label = (\arabic*)]
        \item Although theoretical conditions in literature to establish convergence results  for single-level anisotropic CBO methods suggest choosing $2\lambda>\sigma^2$ when $\kappa=1$ in \cite{fornasier2022anisotropic}, or $\lambda> 16 \sigma^2$ when $\kappa<1$ in \cite[Remark 2.6]{huang2025faithful},  it appears that in practice, one still need $\sigma$ to be large enough. 
        This is mainly due to the difference between the mean--field setting in which the convergence analysis is usually performed, and the particle system which is implemented in practice.  Additionally, the hierarchical structure in multilevel optimization problems further requires increased exploration, motivating the practical choice of larger values for $\sigma$. The need for stronger diffusion has also been observed numerically in related work about saddle point problems \cite{huang2024consensus} and min-max problems \cite{borghi2024particle}. 
        \item As we have already mentioned in Remark \ref{paramater_c}, the parameter $\kappa$ can be set to $1$ in practice, since we do not implement the average system \eqref{ave_model} that is obtained at the limit $\varepsilon\to 0$ (and which require $\kappa<1$). What we do implement instead is an approximation of the averaging procedure which in practice uses \eqref{ms_CBO}, that is for $\varepsilon$ small but positive. 
        \item The choice of $\gamma$ is rather based on empirical observations. It appears that the current consensus point has more accurate information than the previous seemingly because it was computed based on the previous knowledge of the position of the particles. Hence we attribute to the current consensus point a larger weight ($\gamma=0.75$) than to the previous one $(1-\gamma)$ when computing the averaged consensus with the \eqref{FLA} method. This approach is motivated by the more general class of \textit{moving average} methods as known in statistics. See e.g.  \cite{yeh2003multivariate}. 
        \item The choice of the parameter $\alpha$ is motivated by Laplace principle: Letting $\alpha$ and $\beta$ go to infinity enhances the concentration properties of the probability density of the process, as it is analogous to Gibbs measure (where $\alpha,\beta$ would play the role of inverse temperature). Thus we fix it at a very large value to approximate the asymptotic regime (of low temperatures). In practice, the term evolving $\alpha$ and $\beta$ are computed after a shift such as $\exp (-\alpha(F(x, y)-\min_x F))$ and $\exp (-\beta(G(x, y)-\min_y G))$ which better stabilizes the computations.
    \end{enumerate}
     \label{rem:paramaters}
\end{rem}

In \Cref{tb:bilevel} we display the numerical results obtained for bi-level optimization problems, \texttt{MS-CBO} refers to the multiscale CBO algorithm we propose, where \texttt{MS} stands for \textit{multiscale}. The solutions for all the tests are at origin, except for the test (\textcolor{blue}{\texttt{ii}}) whose solution is $x^*=y^*=(1,\cdots,1)$.

\begin{table}[h]
\footnotesize
\begin{tabular}{|ll|ll|}
\hline
\multicolumn{2}{|l|}{\qquad\qquad\qquad\qquad\qquad \textcolor{blue}{$d=10$}}                                                                                                                                                                                                 & \multicolumn{2}{l|}{\qquad\quad\textcolor{blue}{\qquad\texttt{MS-CBO}}}                                           \\ 
\hline
\hline
\multicolumn{1}{|l|}{\multirow{6}{*}{\textcolor{blue}{\textbf{(i)}}}}   & \multirow{6}{*}{\begin{tabular}[c]{@{}l@{}}$F(x,y)=\sum_{i=1}^{d} x_i^2+ y_i^2$\\ $G(x,y)=\sum_{i=1}^{d}(x_i-y_i)^2$\end{tabular}}                 & \multicolumn{1}{l|}{\multirow{2}{*}{success rate}}                 & \multirow{2}{*}{100\%}                 \\
\multicolumn{1}{|l|}{}                                                  &                                                                                                                                                    & \multicolumn{1}{l|}{}                                              &                                        \\ \cline{3-4} 
\multicolumn{1}{|l|}{}                                                  &                                                                                                                                                    & \multicolumn{1}{l|}{\multirow{2}{*}{$\mathbb{E}[\texttt{error}]$}} & \multirow{2}{*}{1.394$\times 10^{-4}$} \\
\multicolumn{1}{|l|}{}                                                  &                                                                                                                                                    & \multicolumn{1}{l|}{}                                              &                                        \\ \cline{3-4} 
\multicolumn{1}{|l|}{}                                                  &                                                                                                                                                    & \multicolumn{1}{l|}{\multirow{2}{*}{running time (s)}}             & \multirow{2}{*}{1.625}                \\
\multicolumn{1}{|l|}{}                                                  &                                                                                                                                                    & \multicolumn{1}{l|}{}                                              &                                        \\ \hline
\hline
\multicolumn{1}{|l|}{\multirow{6}{*}{\textcolor{blue}{\textbf{(ii)}}}}  & \multirow{6}{*}{\begin{tabular}[c]{@{}l@{}}$F(x,y)=\sum_{i=1}^{d} (x_i-1)^2+(y_i-1)^2$\\ $G(x,y)=\sum_{i=1}^{d}(x_i-y_i)^2$\end{tabular}}          & \multicolumn{1}{l|}{\multirow{2}{*}{success rate}}                 & \multirow{2}{*}{100\%}                 \\
\multicolumn{1}{|l|}{}                                                  &                                                                                                                                                    & \multicolumn{1}{l|}{}                                              &                                        \\ \cline{3-4} 
\multicolumn{1}{|l|}{}                                                  &                                                                                                                                                    & \multicolumn{1}{l|}{\multirow{2}{*}{$\mathbb{E}[\texttt{error}]$}} & \multirow{2}{*}{1.353$\times 10^{-4}$} \\
\multicolumn{1}{|l|}{}                                                  &                                                                                                                                                    & \multicolumn{1}{l|}{}                                              &                                        \\ \cline{3-4} 
\multicolumn{1}{|l|}{}                                                  &                                                                                                                                                    & \multicolumn{1}{l|}{\multirow{2}{*}{running time (s)}}             & \multirow{2}{*}{1.616}                \\
\multicolumn{1}{|l|}{}                                                  &                                                                                                                                                    & \multicolumn{1}{l|}{}                                              &                                        \\ \hline
\hline
\multicolumn{1}{|l|}{\multirow{6}{*}{\textcolor{blue}{\textbf{(iii)}}}} & \multirow{6}{*}{\begin{tabular}[c]{@{}l@{}}$F(x,y)=\sum_{i=1}^{d} x_i^2+ y_i^2+2x_i y_i$\\ $G(x,y)=\sum_{i=1}^{d}(x_i-y_i)^2$\end{tabular}} & \multicolumn{1}{l|}{\multirow{2}{*}{success rate}}                 & \multirow{2}{*}{100\%}                 \\
\multicolumn{1}{|l|}{}                                                  &                                                                                                                                                    & \multicolumn{1}{l|}{}                                              &                                        \\ \cline{3-4} 
\multicolumn{1}{|l|}{}                                                  &                                                                                                                                                    & \multicolumn{1}{l|}{\multirow{2}{*}{$\mathbb{E}[\texttt{error}]$}} & \multirow{2}{*}{1.425$\times 10^{-3}$} \\
\multicolumn{1}{|l|}{}                                                  &                                                                                                                                                    & \multicolumn{1}{l|}{}                                              &                                        \\ \cline{3-4} 
\multicolumn{1}{|l|}{}                                                  &                                                                                                                                                    & \multicolumn{1}{l|}{\multirow{2}{*}{running time (s)}}             & \multirow{2}{*}{1.890}                \\
\multicolumn{1}{|l|}{}                                                  &                                                                                                                                                    & \multicolumn{1}{l|}{}                                              &                                        \\ \hline
\hline
\multicolumn{1}{|l|}{\multirow{6}{*}{\textcolor{blue}{\textbf{(iv)}}}}  & \multirow{6}{*}{\begin{tabular}[c]{@{}l@{}}$F(x,y)=A(x)+A(y)$\\ $G(x,y)=\sum_{i=1}^{d}(x_i-y_i)^2$\end{tabular}}                                   & \multicolumn{1}{l|}{\multirow{2}{*}{success rate}}                 & \multirow{2}{*}{100\%}                 \\
\multicolumn{1}{|l|}{}                                                  &                                                                                                                                                    & \multicolumn{1}{l|}{}                                              &                                        \\ \cline{3-4} 
\multicolumn{1}{|l|}{}                                                  &                                                                                                                                                    & \multicolumn{1}{l|}{\multirow{2}{*}{$\mathbb{E}[\texttt{error}]$}} & \multirow{2}{*}{1.333$\times 10^{-4}$} \\
\multicolumn{1}{|l|}{}                                                  &                                                                                                                                                    & \multicolumn{1}{l|}{}                                              &                                        \\ \cline{3-4} 
\multicolumn{1}{|l|}{}                                                  &                                                                                                                                                    & \multicolumn{1}{l|}{\multirow{2}{*}{running time (s)}}             & \multirow{2}{*}{2.682}                \\
\multicolumn{1}{|l|}{}                                                  &                                                                                                                                                    & \multicolumn{1}{l|}{}                                              &                                        \\ \hline
\hline
\multicolumn{1}{|l|}{\multirow{6}{*}{\textcolor{blue}{\textbf{(v)}}}}   & \multirow{6}{*}{\begin{tabular}[c]{@{}l@{}}$F(x,y)=R(x)+R(y)+ 2\sum_{i=1}^dx_iy_i$\\ $G(x,y)=A(x-y)$\end{tabular}}                                                      & \multicolumn{1}{l|}{\multirow{2}{*}{success rate}}                 & \multirow{2}{*}{100\%}                  \\
\multicolumn{1}{|l|}{}                                                  &                                                                                                                                                    & \multicolumn{1}{l|}{}                                              &                                        \\ \cline{3-4} 
\multicolumn{1}{|l|}{}                                                  &                                                                                                                                                    & \multicolumn{1}{l|}{\multirow{2}{*}{$\mathbb{E}[\texttt{error}]$}} & \multirow{2}{*}{4.311$\times 10^{-3}$} \\
\multicolumn{1}{|l|}{}                                                  &                                                                                                                                                    & \multicolumn{1}{l|}{}                                              &                                        \\ \cline{3-4} 
\multicolumn{1}{|l|}{}                                                  &                                                                                                                                                    & \multicolumn{1}{l|}{\multirow{2}{*}{running time (s)}}             & \multirow{2}{*}{3.379}                \\
\multicolumn{1}{|l|}{}                                                  &                                                                                                                                                    & \multicolumn{1}{l|}{}                                              &                                        \\ \hline
\hline
\multicolumn{1}{|l|}{\multirow{6}{*}{\textcolor{blue}{\textbf{(vi)}}}}  & \multirow{6}{*}{\begin{tabular}[c]{@{}l@{}}$F(x,y)=L(x)+L(y)$\\ $G(x,y)=A(x-y)$\end{tabular}}                                                      & \multicolumn{1}{l|}{\multirow{2}{*}{success rate}}                 & \multirow{2}{*}{100\%}                 \\
\multicolumn{1}{|l|}{}                                                  &                                                                                                                                                    & \multicolumn{1}{l|}{}                                              &                                        \\ \cline{3-4} 
\multicolumn{1}{|l|}{}                                                  &                                                                                                                                                    & \multicolumn{1}{l|}{\multirow{2}{*}{$\mathbb{E}[\texttt{error}]$}} & \multirow{2}{*}{1.519$\times 10^{-4}$} \\
\multicolumn{1}{|l|}{}                                                  &                                                                                                                                                    & \multicolumn{1}{l|}{}                                              &                                        \\ \cline{3-4} 
\multicolumn{1}{|l|}{}                                                  &                                                                                                                                                    & \multicolumn{1}{l|}{\multirow{2}{*}{running time (s)}}             & \multirow{2}{*}{4.396}                \\
\multicolumn{1}{|l|}{}                                                  &                                                                                                                                                    & \multicolumn{1}{l|}{}                                              &                                        \\ \hline
\end{tabular}
\captionof{table}{The numerical results
obtained with \texttt{Algorithm \ref{algo averaging}} for bi-level optimization problems \eqref{bilevel opt}, where $n=m=d=10$.\label{tb:bilevel}}
\end{table}
We observe that  \texttt{Algorithm \ref{algo averaging}} demonstrates satisfactory accuracy in solving  bi-level optimization problems, more importantly when the objective functions are non-convex with multiple local minimizers.

\subsection{Tri-level optimization}\label{sec: tri-level opt}
In this section, we consider the tri-level optimization problem \eqref{trilevel opt} with $n=m=p=d$. 
The particles are sampled initially from $\rho_{0}=\mu_{0}=\gamma_0=\mathcal{U}[-1,3]^d$, and the parameters are chosen as
\[
\alpha_1=\alpha_2=\alpha_3=10^{15},\quad  T_x=50, \quad T_y=T_z=0.5,\quad \Delta t=0.1,\quad \kappa=1
\]
\[
N=100, \quad M=50, \quad P=25,\quad \lambda=1 ,\quad \sigma=2,\quad  \gamma=0.75,\quad \delta=10^{-5}, \quad Q=10.
\]
The numerical results obtained for tri-level optimization problems are shown in \Cref{tb:trilevel}. The optimal solution for \textcolor{blue}{\textbf{Test (A)}} and \textcolor{blue}{\textbf{Test (B)}} is $x^*=y^*=r^*=(0,\cdots,0)$ and the optimal solution for \textcolor{blue}{\textbf{Test (C)}} is $x^*=y^*=r^*=(1,\cdots,1)$. 
\begin{table}[!h]
\begin{tabular}{|lll|}
\hline
\multicolumn{3}{|l|}{\textcolor{blue}{\textbf{Test (A)}}}                                                                                                                                                                                                                                                                   \\ \hline
\multicolumn{1}{|l|}{\multirow{6}{*}{\begin{tabular}[c]{@{}l@{}}$F(x,y,r)=\sum_{i=1}^dx_i^2+y_i^2$\\ $G(x,y,r)=L(x-y)$\\ $E(x,y,r)=L(r-y)$\end{tabular}}}                                                     & \multicolumn{1}{l|}{\multirow{2}{*}{success rate}}                 & \multirow{2}{*}{100\%}                 \\
\multicolumn{1}{|l|}{}                                                                                                                                                                                        & \multicolumn{1}{l|}{}                                              &                                        \\ \cline{2-3} 
\multicolumn{1}{|l|}{}                                                                                                                                                                                        & \multicolumn{1}{l|}{\multirow{2}{*}{$\mathbb{E}[\texttt{error}]$}} & \multirow{2}{*}{$1.181\times 10^{-3}$} \\
\multicolumn{1}{|l|}{}                                                                                                                                                                                        & \multicolumn{1}{l|}{}                                              &                                        \\ \cline{2-3} 
\multicolumn{1}{|l|}{}                                                                                                                                                                                        & \multicolumn{1}{l|}{\multirow{2}{*}{running time (s)}}             & \multirow{2}{*}{20.352}                \\
\multicolumn{1}{|l|}{}                                                                                                                                                                                        & \multicolumn{1}{l|}{}                                              &                                        \\ \hline\hline
\multicolumn{3}{|l|}{\textcolor{blue}{\textbf{Test (B)}}}                                                                                                                                                                                                                                                                   \\ \hline
\multicolumn{1}{|l|}{\multirow{6}{*}{\begin{tabular}[c]{@{}l@{}}$F(x,y,r)=\sum_{i=1}^dx_i^2+y_i^2+(r_i-x_i)^2$\\ $G(x,y,r)=L(x-y)$\\ $E(x,y,r)=R(r-y)$\end{tabular}}}                                         & \multicolumn{1}{l|}{\multirow{2}{*}{success rate}}                 & \multirow{2}{*}{94\%}                  \\
\multicolumn{1}{|l|}{}                                                                                                                                                                                        & \multicolumn{1}{l|}{}                                              &                                        \\ \cline{2-3} 
\multicolumn{1}{|l|}{}                                                                                                                                                                                        & \multicolumn{1}{l|}{\multirow{2}{*}{$\mathbb{E}[\texttt{error}]$}} & \multirow{2}{*}{$1.118\times 10^{-1}$} \\
\multicolumn{1}{|l|}{}                                                                                                                                                                                        & \multicolumn{1}{l|}{}                                              &                                        \\ \cline{2-3} 
\multicolumn{1}{|l|}{}                                                                                                                                                                                        & \multicolumn{1}{l|}{\multirow{2}{*}{running time (s)}}             & \multirow{2}{*}{19.092}                \\
\multicolumn{1}{|l|}{}                                                                                                                                                                                        & \multicolumn{1}{l|}{}                                              &                                        \\ \hline\hline
\multicolumn{3}{|l|}{\textcolor{blue}{\textbf{Test (C)}}}                                                                                                                                                                                                                                                                   \\ \hline
\multicolumn{1}{|l|}{\multirow{6}{*}{\begin{tabular}[c]{@{}l@{}}$F(x,y,r)=\sum_{i=1}^d(x_i-1)^2+(y_i-1)^2+(r_i-1)^2$\\ $G(x,y,r)=\sum_{i=1}^d(y_i-x_i)^2$\\ $E(x,y,r)=\sum_{i=1}^d(r_i-y_i)^2$\end{tabular}}} & \multicolumn{1}{l|}{\multirow{2}{*}{success rate}}                 & \multirow{2}{*}{100\%}                 \\
\multicolumn{1}{|l|}{}                                                                                                                                                                                        & \multicolumn{1}{l|}{}                                              &                                        \\ \cline{2-3} 
\multicolumn{1}{|l|}{}                                                                                                                                                                                        & \multicolumn{1}{l|}{\multirow{2}{*}{$\mathbb{E}[\texttt{error}]$}} & \multirow{2}{*}{$2.110\times 10^{-4}$} \\
\multicolumn{1}{|l|}{}                                                                                                                                                                                        & \multicolumn{1}{l|}{}                                              &                                        \\ \cline{2-3} 
\multicolumn{1}{|l|}{}                                                                                                                                                                                        & \multicolumn{1}{l|}{\multirow{2}{*}{running time (s)}}             & \multirow{2}{*}{9.384}                \\
\multicolumn{1}{|l|}{}                                                                                                                                                                                        & \multicolumn{1}{l|}{}                                              &                                        \\ \hline
\end{tabular}
\captionof{table}{The numerical results obtained with \texttt{Algorithm \ref{algo: tri-level}}
 for tri-level optimization problems \eqref{trilevel opt}, where $n=m=p=d=10$.\label{tb:trilevel}}
\end{table}

The results in \Cref{tb:trilevel} show that \texttt{Algorithm \ref{algo: tri-level}} is capable of finding solutions to  tri-level optimization problems and achieving satisfactory accuracy. 

\subsection{Min-max optimization}
\label{sec:min_max}

In this section, we consider the min-max optimization problem (\ref{minmax opt}), where $F:\mathcal{X}\times\mathcal{Y}\rightarrow\mathbb{R}$, with $\mathcal{X}=\mathbb{R}^{d}$ and $\mathcal{Y}=\mathbb{R}^{d}$. We first consider the one dimensional setting to visualize the behaviour of the \texttt{MS-CBO} algorithm.  The  parameters are chosen as
 \begin{equation*}
\begin{aligned}
    & \alpha=\beta=10^{15},\quad T_x=10, \; T_y=2,\quad \Delta t=\Delta \tau=0.1,\quad  N=20,\; M=5,\quad \kappa=1,\\
    &\lambda_1=\lambda_2=1 ,\quad \sigma_1=\sigma_2=1,\quad \gamma=0.75,\quad \delta_1=\delta_2=0.1,\quad R_1=R_2=10.
\end{aligned}
\end{equation*}
  The  $x$- and $y$-particles are sampled initially from $\rho_0=\mu_0=\mathcal{U}[-4,4]$. 
  
In the following figures, the blue points are the initial configuration $(X^i_0,\mathbf{v}^i_{0})_{i=1}^N$, recall the definition of $\mathbf{v}^i_{t}$ in (\ref{v M i t}),   and the red points are the final position of particles $(X^i_{\mathcal{T}_x},\mathbf{v}^i_{\mathcal{T}_y})_{i=1}^N$.  The illustration are shown in Figure \ref{R1}.

\captionsetup[subfigure]{labelformat=empty}
\begin{figure}[h]
\centering
  \begin{subfigure}[t]{0.48\textwidth}
\includegraphics[width=\textwidth]{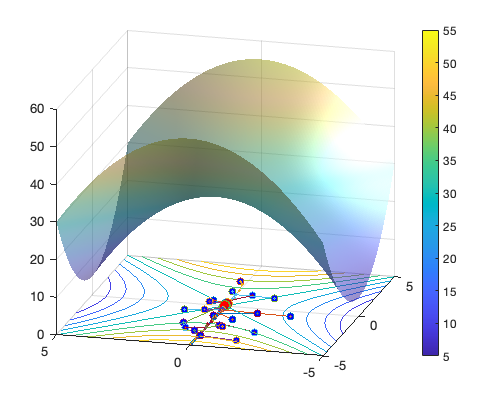}
  \caption{Figure \ref{R1}.1 $F(x,y)=x^2-y^2$}
  \end{subfigure}
\begin{subfigure}[t]
  {0.48\textwidth}
\includegraphics[width=\textwidth]{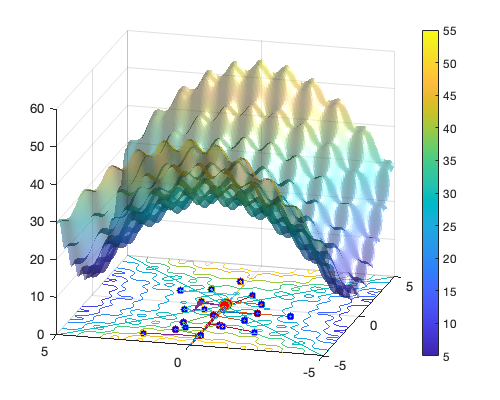}
  \caption{Figure \ref{R1}.2 $F(x,y)=R(x)-R(y)$}
   \end{subfigure}
\begin{subfigure}[b]
  {0.5\textwidth}
\includegraphics[width=1\textwidth]{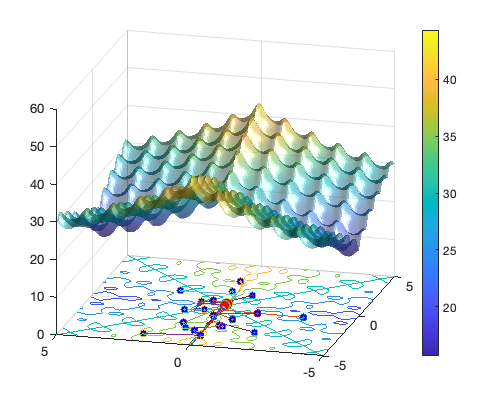}
\caption{Figure \ref{R1}.3 $F(x,y)=A(x)-A(y)$}
\end{subfigure}
\begin{subfigure}[b]
  {0.46\textwidth}
\includegraphics[width=\textwidth]{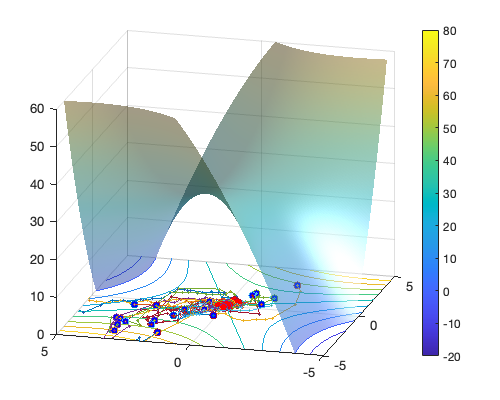}   
 \caption{Figure \ref{R1}.4 $F(x,y)=x^2-y^2-2xy$}
\end{subfigure}
\caption{Illustration of the dynamics for multiscale CBO. The graphs are shifted by $+30$ units.}
\label{R1}
\end{figure}
Then, we further test the performance of \texttt{MS-CBO} by comparing with existing CBO methods, \texttt{SP-CBO} refers to the CBO algorithm proposed in \cite{huang2024consensus}, where \texttt{SP} stands for \textit{saddle points}.

The test functions are defined as
\begin{enumerate}[label=\textcolor{blue}{\textbf{(\alph*)}}]
    \item \textbf{Ackley function:}
    \[
    F(x,y)=A(x)-A(y),
    \]
     \item \textbf{A non-separable (NS) Rastrigin function:}
    \[
    F(x,y)=R(x)-R(y)-2\sum_{i=1}^{d}x_iy_i,
    \]
    \item \textbf{Levy function:} 
    \[
    F(x,y)=L(x)-L(y),
    \]
    \item \textbf{A non-separable function:} 
    \[
    F(x,y)=\sum_{i=1}^{d} x_i^2-\sum_{i=1}^{d} y_i^2-2\sum_{i=1}^{d}x_iy_i.
    \]
\end{enumerate}
The parameters of both \texttt{MS-CBO} and \texttt{SP-CBO} are chosen as in  (\ref{paramaters}).

Note that in \texttt{SP-CBO}, the number of particles in both the $X$- and $Y$-systems are $N$ and time horizon is chosen as $T = T_x = 50$. The numerical experiments for the \texttt{MS-CBO} algorithm are divided into two cases: when $\kappa = 1$ and when $\kappa = 1/\sqrt{M} - 0.01$.
 The numerical results of min-max optimization problems are shown in Table \ref{tb:minmax}. 

\begin{table}[h]
\small
\begin{tabular}{|ll|ll|l|}
\hline
\multicolumn{2}{|l|}{\multirow{2}{*}{\qquad\qquad\quad$d=10$}}                                                                                                           & \multicolumn{2}{l|}{\textcolor{blue}{\qquad\qquad\quad\texttt{MS-CBO}}}  & \multirow{2}{*}{\textcolor{blue}{\quad\texttt{SP-CBO}}} \\ \cline{3-4}
\multicolumn{2}{|l|}{}                                                                                                                                              & \multicolumn{1}{l|}{$\kappa=1$}                 & $\kappa=1/\sqrt{M}-0.01$   &                                                               \\ 
\hline 
\hline 
\multicolumn{1}{|l|}{\multirow{3}{*}{\begin{tabular}[c]{@{}l@{}}Function \textcolor{blue}{\textbf{(a)}}\\ Ackley\end{tabular}}}      & success rate                 & \multicolumn{1}{l|}{100\%}                 & 100\%                 & 99\%                                                          \\ \cline{2-5} 
\multicolumn{1}{|l|}{}                                                                                                               & $\mathbb{E}[\texttt{error}]$ & \multicolumn{1}{l|}{7.689$\times 10^{-5}$} & 1.454$\times 10^{-4}$ & 8.714$\times 10^{-3}$                                         \\ \cline{2-5} 
\multicolumn{1}{|l|}{}                                                                                                               & running time (s)             & \multicolumn{1}{l|}{2.991}                & 2.939               & 0.011                                                         \\ \hline\hline
\multicolumn{1}{|l|}{\multirow{3}{*}{\begin{tabular}[c]{@{}l@{}}Function \textcolor{blue}{\textbf{(b)}}\\ NS-Rastrigin\end{tabular}}} & success rate                 & \multicolumn{1}{l|}{97\%}                  & 77\%                  & 5\%                                                           \\ \cline{2-5} 
\multicolumn{1}{|l|}{}                                                                                                               & $\mathbb{E}[\texttt{error}]$ & \multicolumn{1}{l|}{4.159$\times 10^{-2}$} & 0.245                & 2.597                                                         \\ \cline{2-5} 
\multicolumn{1}{|l|}{}                                                                                                               & running time (s)             & \multicolumn{1}{l|}{3.401}                & 3.107                & 0.013                                                         \\ \hline\hline
\multicolumn{1}{|l|}{\multirow{2}{*}{\begin{tabular}[c]{@{}l@{}}Function \textcolor{blue}{\textbf{(c)}}\\ Levy\end{tabular}}}        & success rate                 & \multicolumn{1}{l|}{100\%}                 & 100\%                 & 99\%                                                          \\ \cline{2-5} 
\multicolumn{1}{|l|}{}                                                                                                               & $\mathbb{E}[\texttt{error}]$ & \multicolumn{1}{l|}{9.007$\times 10^{-5}$} & 1.653$\times 10^{-4}$ & 2.664$\times 10^{-2}$                                         \\ \cline{2-5} 
\multicolumn{1}{|l|}{}                                                                                                               & running time (s)             & \multicolumn{1}{l|}{4.736}                 & 4.734                & 0.015                                                         \\ \hline\hline
\multicolumn{1}{|l|}{\multirow{2}{*}{\begin{tabular}[c]{@{}l@{}}Function \textcolor{blue}{\textbf{(d)}}\\ NS\end{tabular}}}          & success rate                 & \multicolumn{1}{l|}{100\%}                 & 100\%                 & 100\%                                                          \\ \cline{2-5} 
\multicolumn{1}{|l|}{}                                                                                                               & $\mathbb{E}[\texttt{error}]$ & \multicolumn{1}{l|}{1.964$\times 10^{-3}$} & 2.466$\times 10^{-3}$ & 1.474$\times 10^{-3}$                                         \\ \cline{2-5} 
\multicolumn{1}{|l|}{}                                                                                                               & running time (s)             & \multicolumn{1}{l|}{2.019}                & 2.007                & 0.008                                                         \\ \hline
\end{tabular}
\captionof{table}{The numerical results of min-max optimization problems\label{tb:minmax}}
\end{table}

As mentioned in \Cref{rem: model reduction} and \Cref{sec:sensitivity}, we implement the \texttt{for} loop of $y$-particles in parallel with $T_y\ll T_x$.  Although the running time of \texttt{MS-CBO} is larger than the running time of \texttt{SP-CBO}, the advantage is at the level of the numerical complexity, since we have $\mathcal{O}(N^2)$ for \texttt{SP-CBO}, whereas \texttt{MS-CBO} is of $\mathcal{O}(NM)$ with $M\ll N$.  This would be beneficial when the underlying dimension $d$ grows, or when the structure of the optimization problem becomes more complex  (e.g. multi-level problems) for which a larger number of particles would be needed, and consequently a lower numerical complexity would be desirable. 

To further explore the performance of the algorithm \texttt{MS-CBO}, we now choose  different $M$ and $N$, while the other parameters are kept as in the previous test. The results we obtain in this setting are summarized in  \Cref{fig:minmax}.

\begin{figure}[!h]
\centering

\begin{subfigure}[b]{\textwidth}
  \centering
  \subfloat[$d=5$]{
    \includegraphics[width=0.32\textwidth]{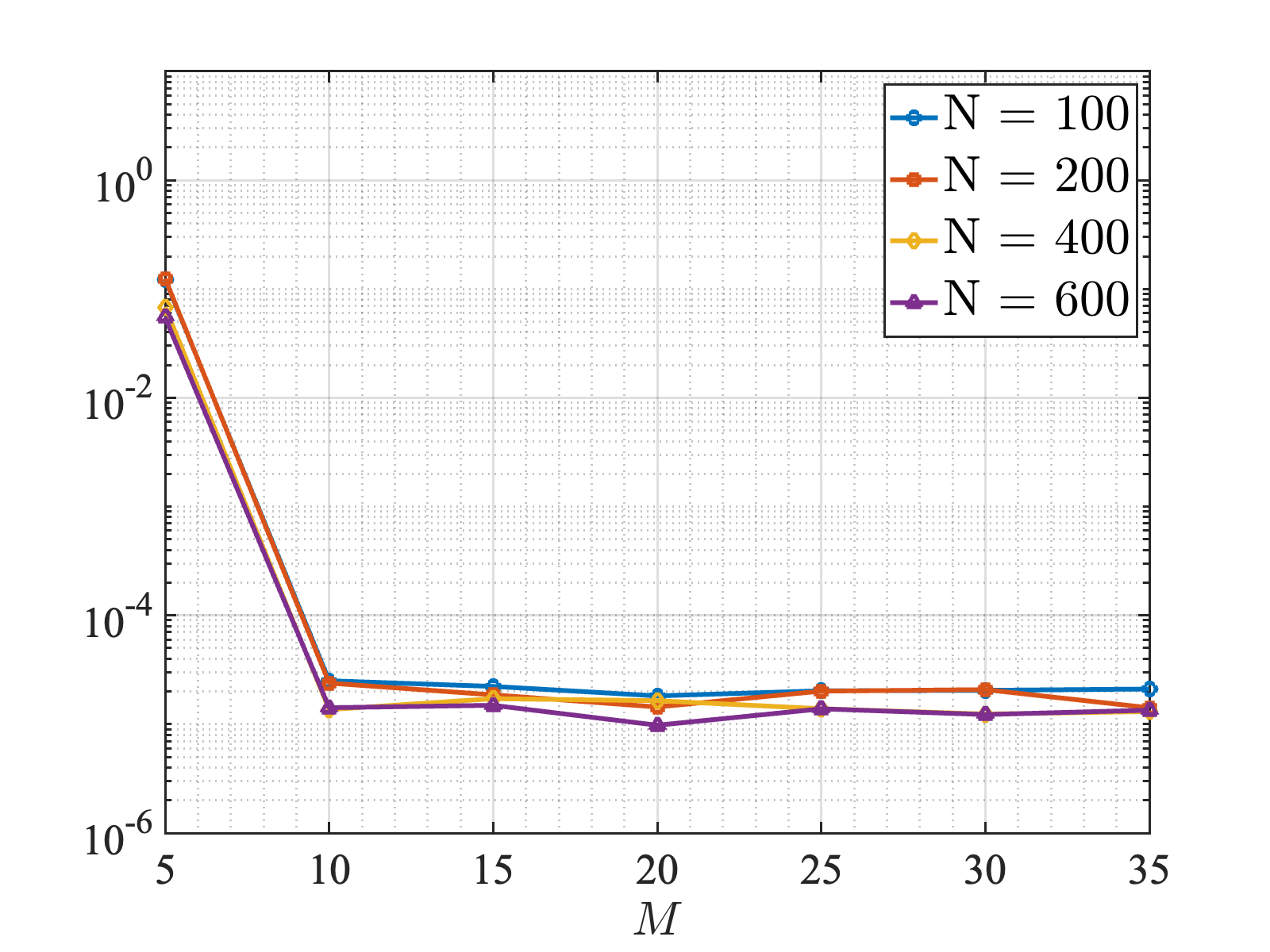}
  }
  \hspace{-0.9em}
  \subfloat[$d=10$]{
    \includegraphics[width=0.32\textwidth]{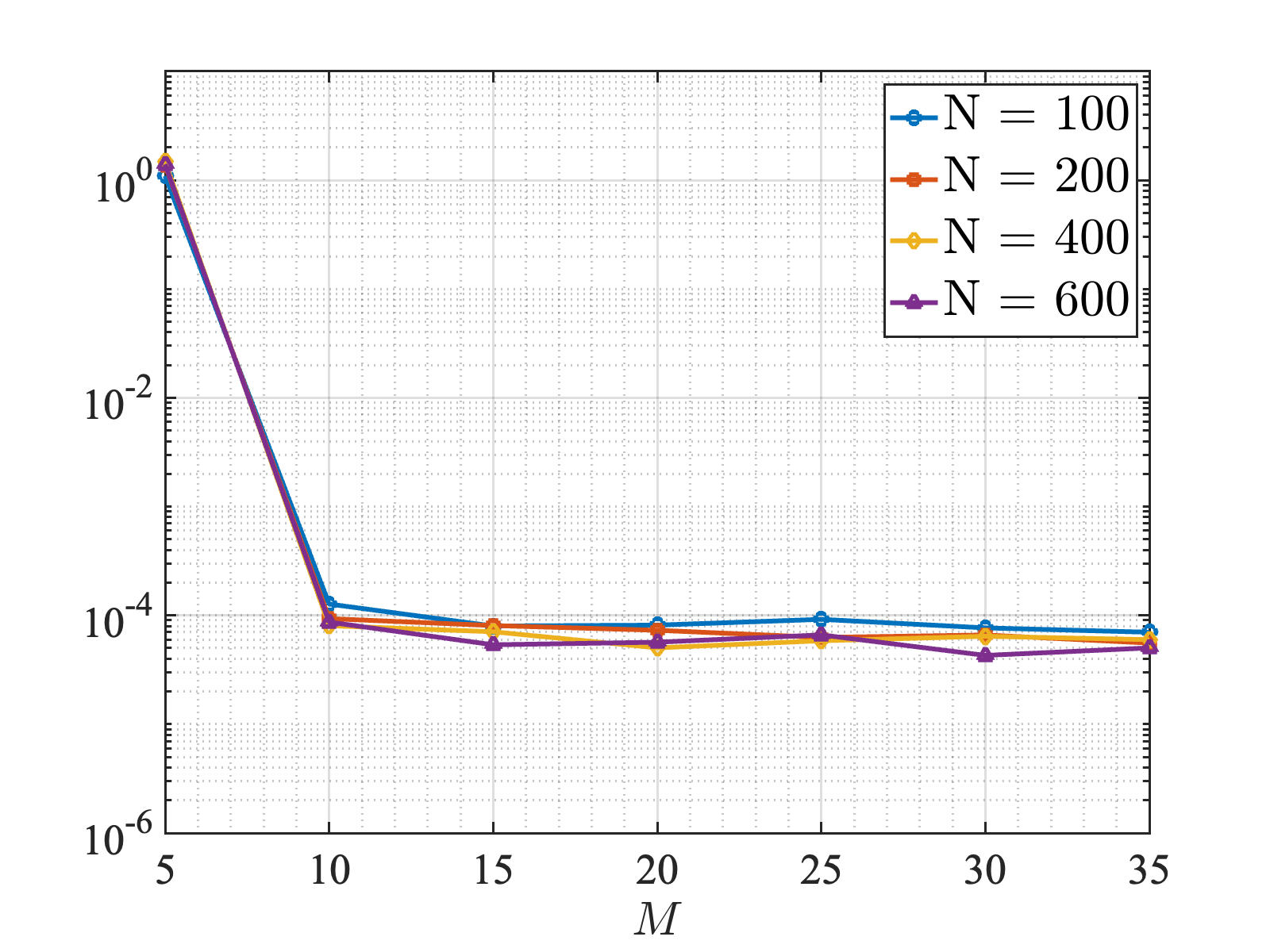}
  }
  \hspace{-0.9em}
  \subfloat[$d=20$]{
    \includegraphics[width=0.32\textwidth]{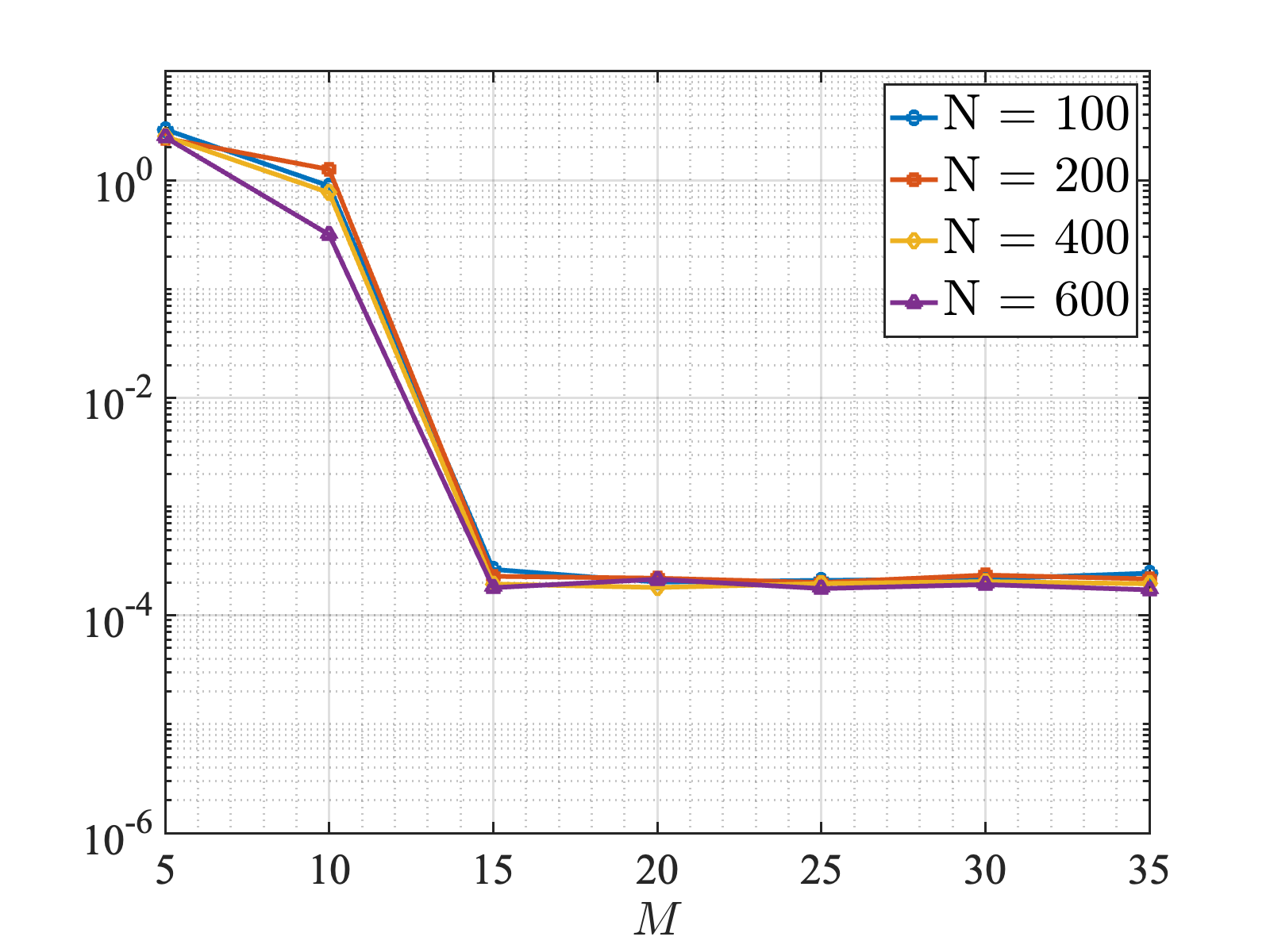}
  }
  \caption{Figure \ref{fig:minmax}.1 Function \textcolor{blue}{\textbf{(a)}}}
\end{subfigure}

\vspace{1em}

\begin{subfigure}[b]{\textwidth}
  \centering
  \subfloat[$d=5$]{
    \includegraphics[width=0.32\textwidth]{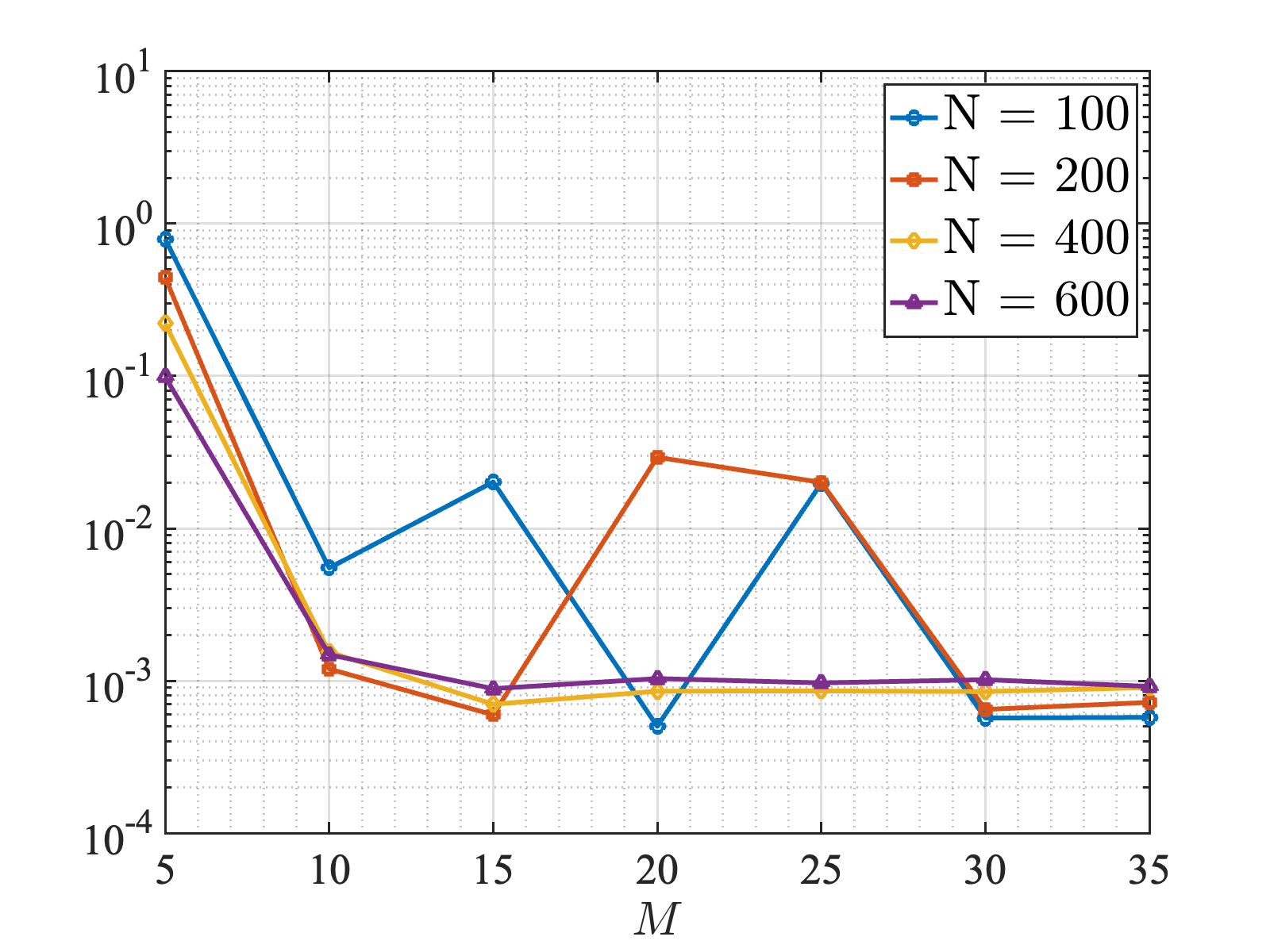}
  }
  \hspace{-0.9em}
  \subfloat[$d=10$]{
    \includegraphics[width=0.32\textwidth]{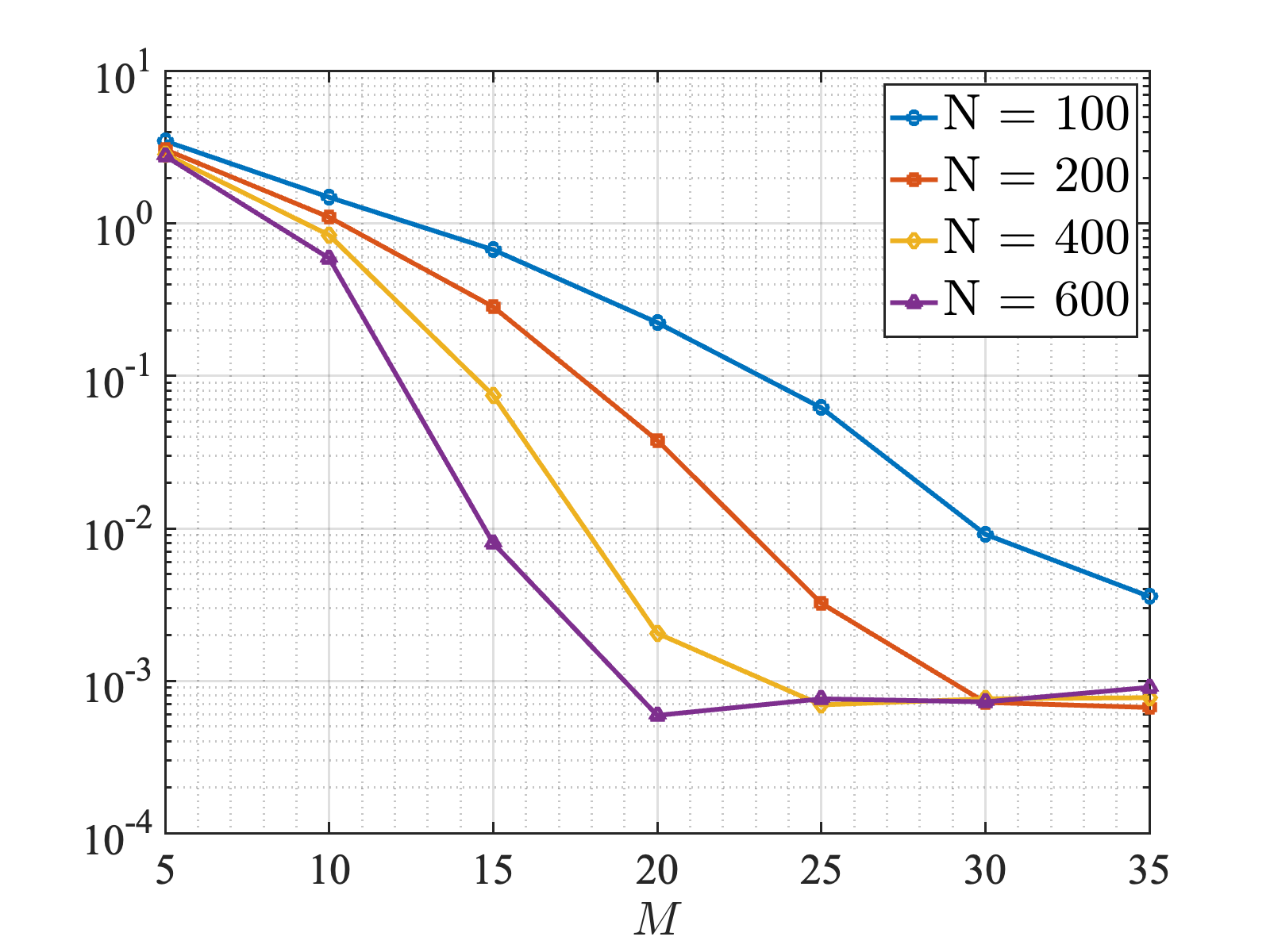}
  }
  \hspace{-0.9em}
  \subfloat[$d=20$]{
    \includegraphics[width=0.32\textwidth]{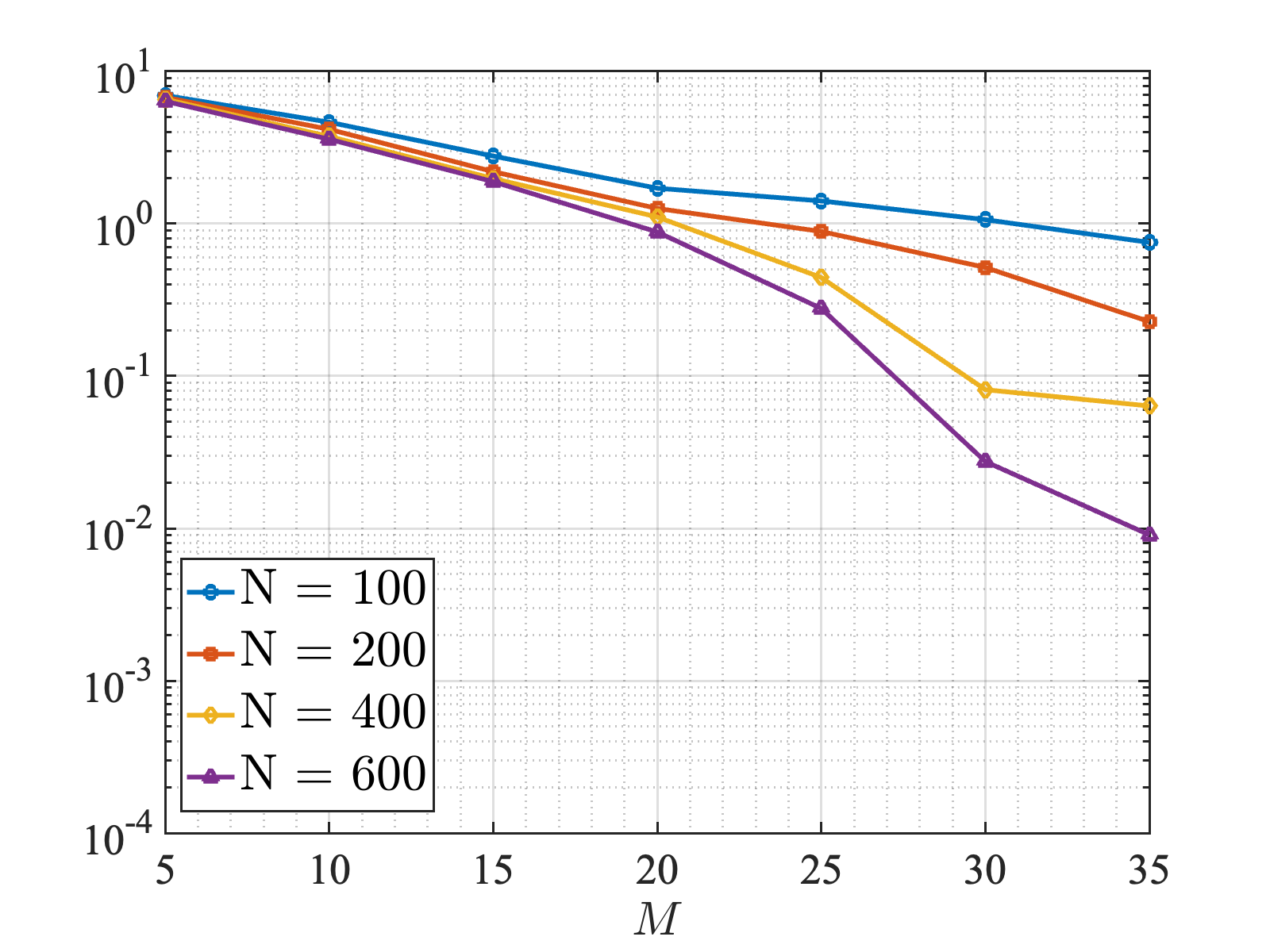}
  }
  \caption{Figure \ref{fig:minmax}.2 Function \textcolor{blue}{\textbf{(b)}}}
\end{subfigure}
\caption{Numerical results of the \texttt{MS-CBO} method with $\kappa=1$ applied to min-max optimization problems. Results are shown for benchmark functions \textcolor{blue}{\textbf{(a)}} (Figure \ref{fig:minmax}.1) and \textcolor{blue}{\textbf{(b)}} (Figure \ref{fig:minmax}.2) with  dimensions $d=5,10,20$. Each subfigure illustrates the $\mathbb{E}[\texttt{error}]$ achieved with different choices of $M$ and $N$.}
\label{fig:minmax}
\end{figure}

Although optimal choices of $M$ and $N$ depend on the specific optimization problem, the numerical results in \Cref{fig:minmax}  suggest that selecting $M\ll N$ is sufficient for solving min-max problems effectively, even in relatively high-dimension. These results will be further investigated in the parameter sensitivity analysis presented in the next section.

\subsection{Parameters sensitivity analysis}\label{sec:sensitivity}

The performance of \texttt{MS-CBO}  depends on the choice of parameters. Therefore it is worth investigating how the parameters $M, N, T_x, T_y, \lambda$, and $\sigma$ impact the accuracy of the algorithm. Figure \ref{fig:chooseT&M} shows this sensitivity. First, we fix $N=100$ and $T_y=0.5$, and examine the influence of changing $M$ and $T_x$ on accuracy for all test functions in \Cref{sec:bilevel}.

\begin{figure}[!h]
    \centering
 \begin{subfigure}[t]{0.48\textwidth}
\includegraphics[width=\textwidth]{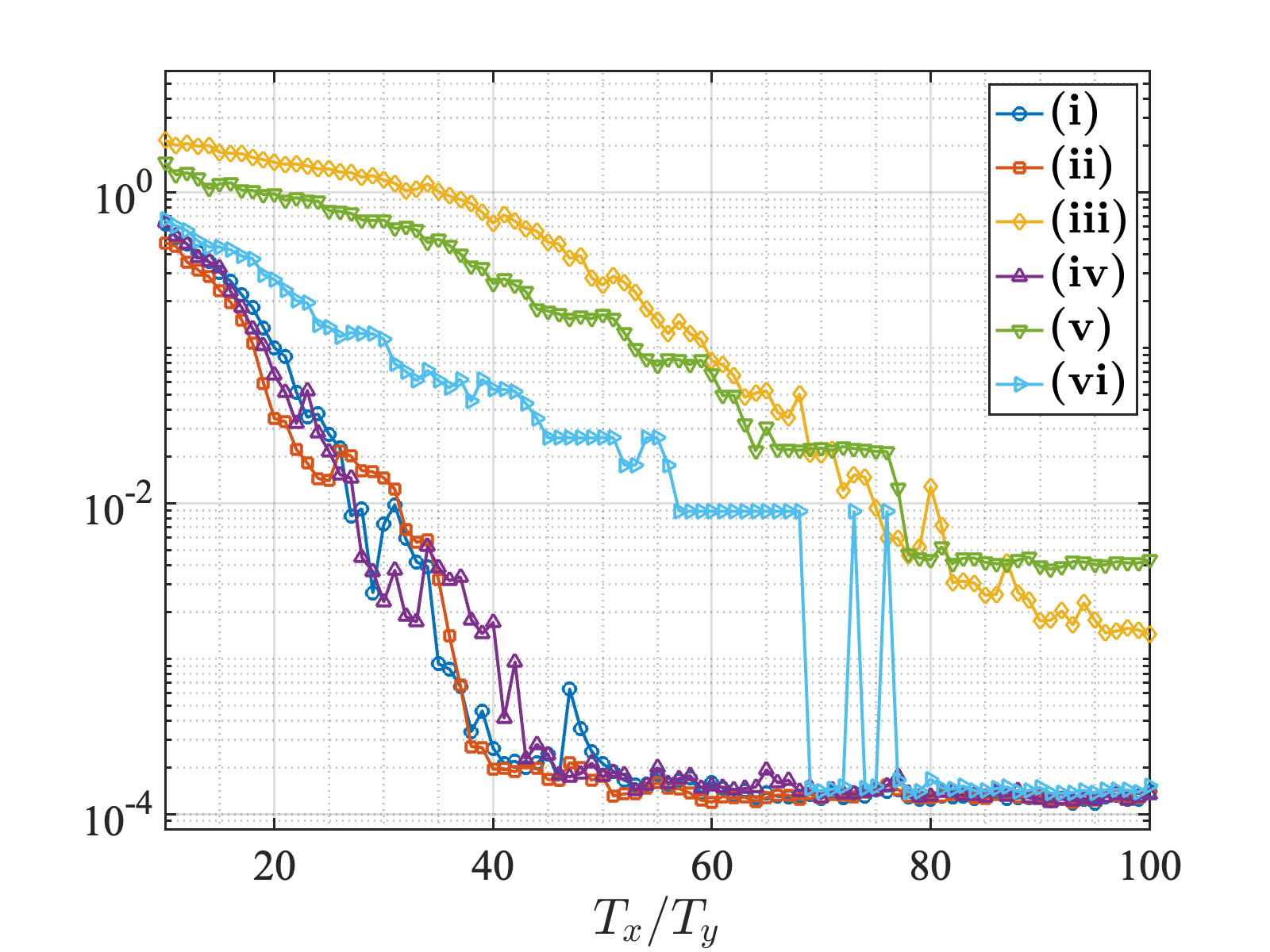}
  \caption{Figure \ref{fig:chooseT&M}.1 Varying parameter $T_x/T_y$}
  \end{subfigure}
\begin{subfigure}[t]
  {0.48\textwidth}
\includegraphics[width=\textwidth]{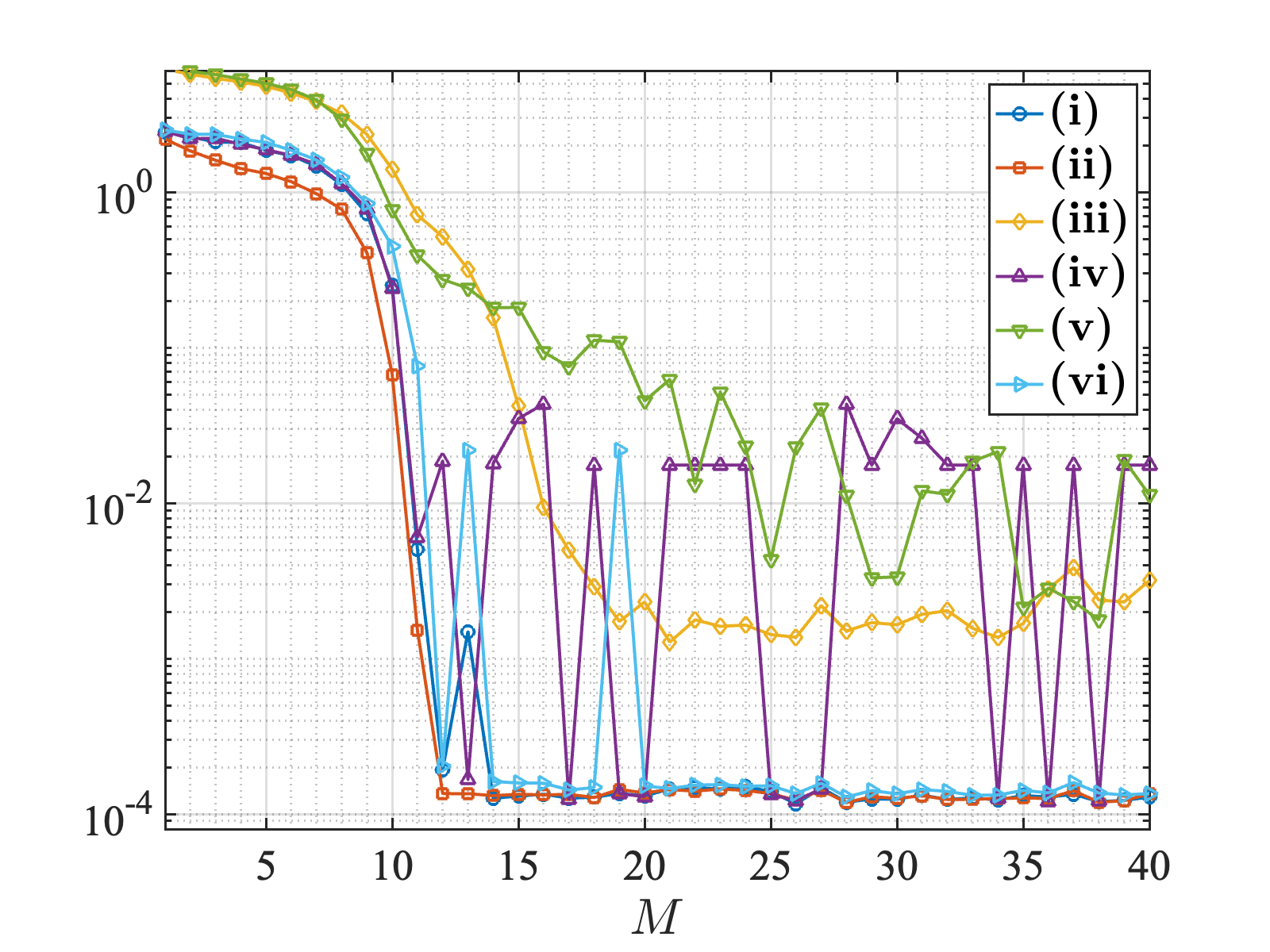}
  \caption{Figure \ref{fig:chooseT&M}.2 Varying parameter $M$}
   \end{subfigure} 
   \begin{subfigure}[t]
  {0.48\textwidth}
\includegraphics[width=\textwidth]{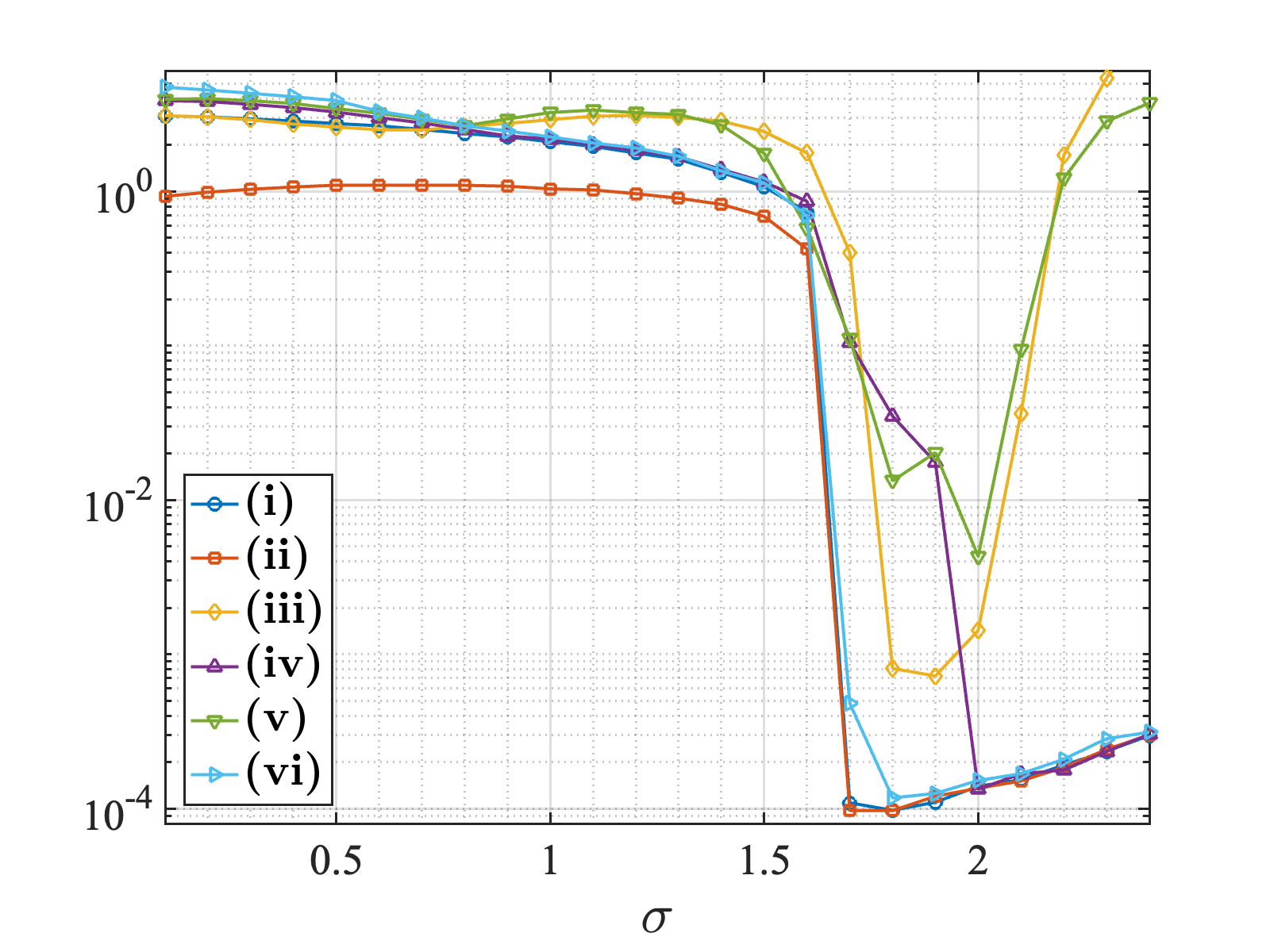}
  \caption{Figure \ref{fig:chooseT&M}.3 Varying parameter $\sigma$}
   \end{subfigure} 
    \caption{$\mathbb{E}[\texttt{error}]$ for benchmark objectives \textcolor{blue}{\textbf{(i)}}-\textcolor{blue}{\textbf{(vi)}} (defined in Table \ref{tb:bilevel})  as a function of the parameter $T_x/T_y$ (Figure \ref{fig:chooseT&M}.1), $M$ (Figure \ref{fig:chooseT&M}.2) and $\sigma$ (Figure \ref{fig:chooseT&M}.3). Only the indicated parameter is varied; all other parameters remain the same as (\ref{paramaters}).}
    \label{fig:chooseT&M}
\end{figure}

As shown in Figure \ref{fig:chooseT&M}.1, accuracy improves with ratio 
$T_x /T_y$ increases, and stabilizes once this ratio is sufficiently large. Similarly, Figure \ref{fig:chooseT&M}.2 is consistent with the earlier observations from the min-max experiments (\Cref{fig:minmax}), 
indicating that $M$ need not be as large as $N$; in most cases, choosing 
$M\geq 20$ ensures satisfactory performance of the \texttt{MS-CBO} method. Note that for benchmark function \textbf{(iv)}, the accuracy exhibits fluctuations as $M$ increases.  This occurs because, despite high success rates (98\%–100\%), a few unsuccessful runs significantly increase the averaged error. However, apart from these rare unsuccessful runs, the algorithm consistently solves the problem effectively.

Another important pair of parameters to choose is $\lambda$ and $\sigma$.  As mentioned in \Cref{rem:paramaters} (1),  in multi-level optimization problems, our empirical observations suggest that the optimal choice of $\sigma$ differs from literature. Specifically, as can be seen in Figure \ref{fig:chooseT&M}.3, when fixing $\lambda=1$, optimal values  of $\sigma$ lie in the interval $[1.8,2.1]$. Values of 
$\sigma$ below this range introduce insufficient randomness, while values above it cause excessive diffusion, both negatively impact the accuracy. All these results  offer practical guidance for selecting parameters in applications of \texttt{MS-CBO} method.

\newpage

\section{Summary and conclusion}\label{sec: conclude}

We have developed a multiscale version of the CBO algorithm,  specifically designed for solving multi-level optimization problems. 
The main idea relies on a singularly perturbed system of SDEs: each scale corresponds to a level in the optimization problem. By using the celebrated averaging principle, this system is reduced  to a single dynamics whose coefficients are averaged with respect to an invariant measure.  The convergence of the original system to the averaged system  has been proved in \Cref{thm: conv}, and the presented algorithms have been tested on various optimization problems with satisfactory results.

Some practical modifications on the CBO model have been made in order to allow the averaging (obtained when $\varepsilon \to 0$) to be proved, in particular using a parameter $\kappa$ to strengthen the drift term. 
In \Cref{thm: conv}, this parameter is assumed to be small in order to guarantee the averaging. However, this is not a compulsory requirement in numerical implementation.  Therefore, a sharper convergence analysis for broader choice of $\kappa$ could be worth investigating.   Another natural direction for further investigation is to explore the convergence guarantees of mean-field dynamics to the global solution of the optimization problem. We can study the mean-field limit of the averaged (effective) model (\ref{ave_model}) as $N\rightarrow \infty$,  similar to  the theoretical analysis of standard CBO methods.  Yet, this is a challenging topic as difficulty arises from the invariant measure and its regularity with respect to the $x$-variable. On the other hand, the criteria of choosing optimal parameters in different dynamics remain to be understood, e.g., a more precise relationship is to be established between $M$ and  $N$, $T_x$ and $T_y$ in bi-level problems.  Lastly, in both Algorithms \ref{algo averaging} and \ref{algo: tri-level}, we  used  the approximation \eqref{FLA} to handle the averaging with respect to the invariant measure. But we believe the implementation of the two algorithms could be further improved if one improves the computation of the invariant measure or the averaging with its respect.

\subsection*{Acknowledgments}
The authors wish to thank the two referees for their valuable comments which helped improve the paper.

\appendix

\section{Proof of Theorem \ref{thm: conv}}\label{app: proof conv}

We start by defining some functional spaces that we shall use. 
\begin{itemize}
    \item If $\theta \in (0,1)$, then $C^{\theta}_{b}(\mathbb{R}^{d})$ denotes the usual space of locally H\"older continuous functions with exponent $\theta$, and that are moreover bounded. 
    \item If $\theta = 1$, then $C^{1}_{b}(\mathbb{R}^{d})$ denotes the space of bounded and Lipschitz continuous functions. 
    \item If $\theta \in (1,2)$, then $C^{\theta}_{b}(\mathbb{R}^{d})$ denotes the space of functions that are in $C^{1}_{b}(\mathbb{R}^{d})$ and whose first-order derivatives are in $C^{\theta-1}_{b}(\mathbb{R}^{d})$, noting that $\theta-1 \in (0,1)$.
    \item If $\theta = 2$, then $C^{2}_{b}(\mathbb{R}^{d})$ denotes the space of bounded functions whose first-order derivatives are Lipschitz continuous. 
\end{itemize}
Now let $(x,y)\in \mathbb{R}^{d_{1}}\times \mathbb{R}^{d_{2}}$ for some $d_{1},d_{2}\in \mathbb{N}$, and let $f=f(x,y)$ be a function of $(x,y)$. We denote by $C_{b}^{\theta, \delta} := C_{b}^{\theta}(\mathbb{R}^{d_{1}}; C_{b}^{\delta}(\mathbb{R}^{d_{2}}))$ the space functions $f$ such that $x\mapsto f(x,y) \in C_{b}^{\theta}(\mathbb{R}^{d_{1}})$ locally uniformly in $y$, and $y \mapsto f(x,y) \in C_{b}^{\delta}(\mathbb{R}^{d_{2}})$ locally uniformly in $x$.

\begin{proof}[Proof of Theorem \ref{thm: conv}]
This shall be a consequence of statement $(i)$ in \cite[Theorem 2.3]{rockner2021diffusion} or, more precisely, of the particular case discussed in statement $(i)$ of Remark 2.5 therein. (A similar result can also be found in \cite[Theorem 2.5]{rockner2019strong}.) \\
Therefore, it suffices to check that our system of coupled SDEs \eqref{ms_CBO} satisfies the assumptions in \cite{rockner2021diffusion}. To this end, and for better readability, we first highlight the correspondence between our notation and the one in \cite{rockner2021diffusion} as follows:
\begin{equation*}
\begin{aligned}
    & \text{\textbullet\; in \cite{rockner2021diffusion}}: &   X^{\varepsilon},\; & Y^{\varepsilon},\; & b,\; & \sigma,\; & F,\; & G,\; & c,\; & H,\; & \alpha_{\varepsilon}.\\
    & \text{\textbullet\; in \eqref{ms_CBO}}: & \mathds{Y}^{\varepsilon}\,,\; & \mathds{X}^{\varepsilon},\; & \mathscr{B},\; & \mathscr{H},\; & \mathscr{F},\; & \mathscr{G},\; & 0,\; & \,0,\; & \sqrt{\varepsilon}.
\end{aligned}
\end{equation*}
We can now state the assumptions needed, and prove that they are satisfied. 

\begin{enumerate}[label = (\roman*)]
    \item $\mathscr{H}\mathscr{H}^{*}$ is bounded and non-degenerate in $y$, uniformly with respect to $x$. \textit{(This corresponds to assumptions $(A_\sigma)$ in \cite{rockner2021diffusion}.)}
    \item $\mathscr{G}\mathscr{G}^{*}$ is bounded and non-degenerate in $x$, uniformly with respect to $y$. \textit{(This corresponds to assumptions $(A_G)$ in \cite{rockner2021diffusion}.)}
    \item The drift of the fast process satisfies a recurrence assumption which ensures the existence of an invariant measure $\Upsilon_{x}(\text{d}y)$, and that is 
    \begin{equation*}
        \lim\limits_{|y|\to +\infty} \, \sup\limits_{x} \, \langle y, \, \mathscr{B}(x,y) \rangle = -\infty.
    \end{equation*}
    \textit{(This corresponds to assumptions $(A_b)$ in \cite{rockner2021diffusion}.)}
    \item The coefficients $\mathscr{F},\mathscr{G}$ (of the slow variables $\mathds{X}^{\varepsilon}$), and $\mathscr{B},\mathscr{H}$ (of the fast variables $\mathds{Y}^{\varepsilon}$) are in $C_{b}^{\theta,\delta}$, with $\theta\in (0,1)$ and $\delta \in (0,1)$. \textit{(This is an assumption slightly stronger than the regularity requirement in \cite{rockner2021diffusion}.)}
\end{enumerate}

Our construction of $\mathscr{H}$ and $\mathscr{G}$ in \S \ref{sec: MS CBO} guarantee the validity of assumptions (i) and (ii). To check assumption (iii), we proceed as follows. Given $\mathbf{x}=(x_{i})_{1\leq i \leq N}$, where $x_{i} \in \mathbb{R}^{n}$, and setting $\mathds{Y}^{\varepsilon}=:\mathds{Y}$ where
\begin{equation*}
    \mathds{Y} = (\mathbf{Y}_{i})_{1\leq i \leq N} \in \mathbb{R}^{mMN},\quad \text{ where } \mathbf{Y}_{i} = (Y^{j}_{x_{i}})_{1\leq j \leq M} \in \mathbb{R}^{mM}, \quad \text{ and } Y^{j}_{x_{i}} \in \mathbb{R}^{m},
\end{equation*}
with $Y^{j}_{x_{i}} = (Y^{j}_{x_{i}}(k))_{1\leq k \leq m}$, 
we have 
\begin{equation*}
    \mathscr{B}(\mathbf{x},\mathds{Y}) = \big( B(x_{i},\mathbf{Y}_{i}) \big)_{1\leq i \leq N}, \quad \text{ where } \quad 
    B(x_{i},\mathbf{Y}_{i}) = \big( b_{i}^{j}(x_{i},\mathbf{Y}_{i}) \big)_{1\leq j \leq M}
\end{equation*}
and 
\begin{equation*}
\begin{aligned}
    b_{i}^{j}(x_{i},\mathbf{Y}_{i}) 
    & = -\lambda_{2} \, \psi_{R_{2}} \left(\,Y_{x_{i}}^j - \kappa\,\sum_{r=1}^M Y_{x_{i}}^r\frac{\varpi_{\beta}\left(x_i, Y_{x_{i}}^r\right)}{\sum_{\ell=1}^M \varpi_{\beta}\left(x_i, Y_{x_{i}}^{\ell}\right)}\,\right)\in\mathbb{R}^{m}
\end{aligned}
\end{equation*}
using the truncation function defined, for a vector $v\in \mathbb{R}^{d}$, by 
\begin{equation*}
    \psi_{R}(v) = \big(\psi_{R}(v)_{k}\big)_{1\leq k\leq d} \quad \text{ where } \psi_{R}(v)_{k} = \left\{\;
    \begin{aligned}
        & v_{k} \; , \quad \quad \text{ if } \; |v_{k}|\leq R\\
        & \frac{v_{k}}{|v_{k}|}R \; , \quad \text{ if } \; |v_{k}|>R.
    \end{aligned}
    \right.
\end{equation*}
Then we obtain
\begin{equation}\label{eq: scalar prod proof}
\begin{aligned}
    \langle\, \mathds{Y}, \, \mathscr{B}(x,\mathds{Y})\, \rangle 
    = \sum\limits_{i=1}^{N} \; \langle \, \mathbf{Y}_{i}, \, B(x_{i}, \, \mathbf{Y}_{i}) \, \rangle = \sum\limits_{i=1}^{N}\sum\limits_{j=1}^{M} \; \langle \,  Y_{x_{i}}^{j}, \, b_{i}^{j}(x_{i}, \mathbf{Y}_{i})\, \rangle
\end{aligned}
\end{equation}
and 
\begin{equation*}
\begin{aligned}
    \langle \,  Y_{x_{i}}^{j}, \, b_{i}^{j}(x_{i}, \mathbf{Y}_{i})\, \rangle 
    & = -\lambda_{2} \, \langle \, Y_{x_{i}}^{j} ,\, \psi_{R_{2}}\left(\,Y_{x_{i}}^j - \kappa\,\sum_{r=1}^M Y_{i}^r \frac{\varpi_{\beta}\left(x_i, Y_{x_{i}}^r \right)}{\sum_{\ell=1}^M \varpi_{\beta}\left(x_i, Y_{x_{i}}^{\ell}\right)}\,\right) \, \rangle.
\end{aligned}
\end{equation*}
Notice that, for any $w,v\in \mathbb{R}^{d}$, we have $\langle w,\psi_{R_{2}}(v) \rangle = \sum_{\tau=1}^{d} \langle w,\psi_{R_{2}}(v) \rangle_{\tau}$ and
\begin{equation*}
    \langle w,\psi_{R_{2}}(v) \rangle_{\tau} =
    \left\{\;
    \begin{aligned}
        & w_{\tau}v_{\tau} \; , \quad \quad \text{ if } \; |v_{\tau}|\leq R_{2},\\
        & w_{\tau}v_{\tau}\frac{R}{|v_{\tau}|} \; , \quad \text{ if } \; |v_{\tau}|>R_{2},
    \end{aligned}
    \right.
    \quad \tau=1, \dots, d.
\end{equation*}
We can then write $\langle w,\psi_{R_{2}}(v) \rangle_{\tau} = w_{\tau}v_{\tau}\,\xi(v_{\tau})$ where $\xi(v_{\tau}) := 1 \wedge \frac{R_{2}}{|v_{\tau}|}$, and $\xi(0)=1$. Note in particular that $0<\xi(v_{\tau})\leq 1$ for all $v_{\tau}\in\mathbb{R}$. 
Therefore we have, for $\tau=1, \dots, m$
\begin{equation*}
\begin{aligned}
    & \langle \,  Y_{x_{i}}^{j}, \, b_{i}^{j}(x_{i}, \mathbf{Y}_{i})\, \rangle_{\tau} \\
    & \quad \quad =  
    \bigg( -\lambda_{2} \, Y_{x_{i}}^{j}(\tau)^{2} +\lambda_{2} \,\kappa \; \sum_{r=1}^M Y_{x_{i}}^{j}(\tau)Y_{x_{i}}^{r}(\tau)\frac{\varpi_{\beta}\left(x_i, Y_{x_{i}}^r\right)}{\sum_{\ell=1}^M \varpi_{\beta}\left(x_i, Y_{x_{i}}^{\ell}\right)}\,\bigg)\\
    & \quad \quad \quad \quad \quad \quad \quad  \quad \quad \times \xi \bigg(\,Y_{x_{i}}^{j}(\tau) - \kappa \,\sum_{r=1}^M Y_{x_{i}}^{r}(\tau) \frac{\varpi_{\beta}\left(x_i, Y_{x_{i}}^r\right)}{\sum_{\ell=1}^M \varpi_{\beta}\left(x_i, Y_{x_{i}}^{\ell}\right)} \, \bigg) .
\end{aligned}
\end{equation*}
Let us denote by $\underline{\xi}$ the following
\begin{equation*}
    \underline{\xi}(\mathds{Y}) := \min\limits_{\substack{1\leq \tau\leq m \\ 1\leq j \leq M \\ 1\leq i \leq N}}\left\{ \xi \bigg(\,Y_{x_{i}}^{j}(\tau) - \kappa \,\sum_{r=1}^M Y_{x_{i}}^{r}(\tau) \frac{\varpi_{\beta}\left(x_i, Y_{x_{i}}^r\right)}{\sum_{\ell=1}^M \varpi_{\beta}\left(x_i, Y_{x_{i}}^{\ell}\right)} \, \bigg)  \right\} \, \in (0,1].
\end{equation*}
We have
\begin{equation}\label{eq: tau coordinate}
\begin{aligned}
    \langle \,  Y_{x_{i}}^{j}, \, b_{i}^{j}(x_{i}, \mathbf{Y}_{i})\, \rangle_{\tau} & \leq -\lambda_{2} \, Y_{x_{i}}^{j}(\tau)^{2} \, \underline{\xi}(\mathds{Y})  +\lambda_{2} \,\kappa\,\Theta_{i}^{j}(\tau)
\end{aligned}
\end{equation}
where 
\begin{equation*}
\begin{aligned}
    \Theta_{i}^{j}(\tau) & := \bigg(\sum_{k=1}^M Y_{x_{i}}^{j}(\tau)Y_{x_{i}}^{k}(\tau)\frac{\varpi_{\beta}\left(x_i, Y_{x_{i}}^k\right)}{\sum_{\ell=1}^M \varpi_{\beta}\left(x_i, Y_{x_{i}}^{\ell}\right)}\,\bigg)\\
    & \quad \quad \quad \times\,\xi \bigg(\,Y_{x_{i}}^{j}(\tau) - \kappa \,\sum_{k=1}^M Y_{x_{i}}^{k}(\tau) \frac{\varpi_{\beta}\left(x_i, Y_{x_{i}}^k\right)}{\sum_{\ell=1}^M \varpi_{\beta}\left(x_i, Y_{x_{i}}^{\ell}\right)} \, \bigg).
\end{aligned}
\end{equation*}
Taking the absolute value, and using $|\xi(\cdot)|\leq 1$, one gets
\begin{equation*}
\begin{aligned}
    \big|\Theta_{i}^{j}(\tau)\big| & \leq \sum_{k=1}^{M} |Y_{x_{i}}^{j}(\tau)|\,|Y_{x_{i}}^{k}(\tau)|\,\frac{\varpi_{\beta}\left(x_i, Y_{x_{i}}^k\right)}{\sum_{\ell=1}^M \varpi_{\beta}\left(x_i, Y_{x_{i}}^{\ell}\right)},
\end{aligned}
\end{equation*}
and then 
\begin{equation*}
\begin{aligned}
    \sum\limits_{\tau=1}^{m} \big|\Theta_{i}^{j}(\tau)\big| & \leq \sum_{k=1}^{M} \sum\limits_{\tau=1}^{m} |Y_{x_{i}}^{j}(\tau)|\,|Y_{x_{i}}^{k}(\tau)|\,\frac{\varpi_{\beta}\left(x_i, Y_{x_{i}}^k\right)}{\sum_{\ell=1}^M \varpi_{\beta}\left(x_i, Y_{x_{i}}^{\ell}\right)}.
\end{aligned}
\end{equation*}
Using Cauchy–Schwarz inequality yields
\begingroup\allowdisplaybreaks
\begin{align*}
    \sum\limits_{\tau=1}^{m} \big|\Theta_{i}^{j}(\tau)\big| & \leq \sum_{k=1}^{M} \left(\sum\limits_{\tau=1}^{m} |Y_{x_{i}}^{j}(\tau)|^{2} \right)^{\frac{1}{2}}
    \left(\sum\limits_{\tau=1}^{m} |Y_{x_{i}}^{k}(\tau)|^{2} \right)^{\frac{1}{2}} \frac{\varpi_{\beta}\left(x_i, Y_{x_{i}}^k\right)}{\sum_{\ell=1}^M \varpi_{\beta}\left(x_i, Y_{x_{i}}^{\ell}\right)}\\
    & = \sum_{k=1}^{M} \|Y_{x_{i}}^{j}\| \,\|Y_{x_{i}}^{k}\| \,\frac{\varpi_{\beta}\left(x_i, Y_{x_{i}}^k\right)}{\sum_{\ell=1}^M \varpi_{\beta}\left(x_i, Y_{x_{i}}^{\ell}\right)}
     \\
    & \leq \left( \sum_{k=1}^{M} \|Y_{i}^{j}\|^{2} \,\frac{\varpi_{\beta}\left(x_i, Y_{x_{i}}^k\right)}{\sum_{\ell=1}^M \varpi_{\beta}\left(x_i, Y_{x_{i}}^{\ell}\right)}\right)^{\frac{1}{2}} 
    \left( \sum_{k=1}^{M} \|Y_{x_{i}}^{k}\|^{2} \,\frac{\varpi_{\beta}\left(x_i, Y_{x_{i}}^k\right)}{\sum_{\ell=1}^M \varpi_{\beta}\left(x_i, Y_{x_{i}}^{\ell}\right)}\right)^{\frac{1}{2}}\\
    & =  \|Y_{x_{i}}^{j}\| \left( \sum_{k=1}^{M} \|Y_{x_{i}}^{k}\|^{2} \,\frac{\varpi_{\beta}\left(x_i, Y_{x_{i}}^k\right)}{\sum_{\ell=1}^M \varpi_{\beta}\left(x_i, Y_{x_{i}}^{\ell}\right)}\right)^{\frac{1}{2}}\\
    & \leq  \|Y_{x_{i}}^{j}\| \left( \sum_{k=1}^{M} \|Y_{x_{i}}^{k}\|^{2} \right)^{\frac{1}{2}} = \|Y_{x_{i}}^{j}\| \, \|\mathbf{Y}_{i}\|.
\end{align*}
\endgroup
The latter inequality together with \eqref{eq: tau coordinate} yield
\begin{equation*}
\begin{aligned}
    \langle \,  Y_{x_{i}}^{j}, \, b_{i}^{j}(x_{i}, \mathbf{Y}_{i})\, \rangle & = \sum\limits_{\tau = 1}^{m} \langle \,  Y_{x_{i}}^{j}, \, b_{i}^{j}(x_{i}, \mathbf{Y}_{i})\, \rangle_{\tau} \\
    & \leq -\lambda_{2}\,\underline{\xi}(\mathds{Y}) \, \|Y_{x_{i}}^{j}\|^{2} + \lambda_{2}\,\kappa\,  \|Y_{x_{i}}^{j}\| \, \|\mathbf{Y}_{i}\|,
\end{aligned}
\end{equation*}
and with \eqref{eq: scalar prod proof}, one obtains (using again Cauchy-Schwarz inequality)
\begin{equation*}
\begin{aligned}
    \langle\, \mathds{Y}, \, \mathscr{B}(x,\mathds{Y})\, \rangle 
    & = \sum\limits_{i=1}^{N}\sum\limits_{j=1}^{M} \; \langle \,  Y_{x_{i}}^{j}, \, b_{i}^{j}(x_{i}, \mathbf{Y}_{i})\, \rangle \\
    & \leq -\lambda_{2}\,\underline{\xi}(\mathds{Y}) \, \|\mathds{Y}\|^{2} + \lambda_{2}\,\kappa\,   \sum\limits_{i=1}^{N}\sum\limits_{j=1}^{M} \|Y_{x_{i}}^{j}\| \, \|\mathbf{Y}_{i}\|\\
    & \leq -\lambda_{2}\,\underline{\xi}(\mathds{Y}) \, \|\mathds{Y}\|^{2} + \lambda_{2}\,\kappa\,   \sum\limits_{i=1}^{N}\big(\|\mathbf{Y}_{i}\|\big) \, \big(\sqrt{M}\,\|\mathbf{Y}_{i}\|\big)\\
    & = -\lambda_{2}\,\underline{\xi}(\mathds{Y}) \, \|\mathds{Y}\|^{2} + \lambda_{2}\,\kappa\,\sqrt{M}\,  \|\mathds{Y}\|^{2} 
    = -\lambda_{2}\big( \underline{\xi}(\mathds{Y}) - \kappa\sqrt{M} \big)\|\mathds{Y}\|^{2}.
\end{aligned}
\end{equation*}
Therefore, in order to satisfy assumption (iii), it suffices to have 
\begin{equation*}
    \kappa < \frac{\underline{\xi}(\mathds{Y})}{\sqrt{M}}.
\end{equation*}
Let us note that $\underline{\xi}(\mathds{Y})$ tends to $1$ when the particles $Y_{x_{i}}^{j}$ tends to a consensus, and $\underline{\xi}(\mathds{Y})$ tends to $0$ only when a $Y_{x_{i}}^{j}$ gets (infinitely) far from the others, the latter being a behavior that is not likely to happen, provided we choose the radius $R_{2}$ of the truncation large enough compared to the closed ball containing the initial position of the particles. This will prevent $\xi(\mathds{Y})$ from getting smaller, allowing $\kappa$ to exist. \\
Finally, the coefficients in \eqref{ms_CBO} are bounded and locally H\"older continuous\footnote{In fact, they are locally Lipschitz continuous as can be seen from the results in \cite{huang2024consensus}.} with any exponent in the open interval $(0,1)$, hence the validity of assumption (iv). \\
Thus we can apply the result in \cite{rockner2021diffusion} with coefficients that are bounded in both $x,y$, and locally H\"older continuous with exponents $\theta,\delta\in (0,1)$, which concludes the proof. 
\end{proof}

\bibliography{bibliography}
\bibliographystyle{plain}

\end{document}